\documentclass[12pt]{article}
\title{On the pair correlation function of the $\Sineb$ process}
\usepackage{palatino}
\date{}
\author{Yahui Qu and Benedek Valk\'o}

\usepackage{amsmath, amsthm, amssymb}
\usepackage{graphicx}
\usepackage{amsmath, enumerate}
\usepackage{color, url}

\usepackage{hyperref}
    \newtheorem{theorem}{Theorem}
    \newtheorem{lemma}[theorem]{Lemma}
    \newtheorem{proposition}[theorem]{Proposition}
    \newtheorem{corollary}[theorem]{Corollary}
    \newtheorem{claim}[theorem]{Claim}

\theoremstyle{definition} 
    \newtheorem{definition}[theorem]{Definition}

    \newtheorem{problem}{Problem}

\newcommand{\eps}{\varepsilon}
\newcommand{\Z}{{\mathbb Z}}
\newcommand{\ZZ}{{\mathbb Z}}

\newcommand{\R}{{\mathbb R}}
\newcommand{\CC}{{\mathbb C}}

\newcommand{\lstar}{{\raise-0.15ex\hbox{$\scriptstyle \ast$}}}

\theoremstyle{remark} 
\newcommand{\Sineb}{\operatorname{Sine}_{\beta}}
\newcommand{\Sine}{\operatorname{Sine}}

\newcommand{\HPb}{\operatorname{HP}_{\beta, \delta}}
\newcommand{\HPbb}{\operatorname{HP}_{\beta, \beta/2}}

\definecolor{violet}{rgb}{0.8,0,0.2}

\newcommand{\ed}{\stackrel{d}{=}}

\newcommand{\cB}{{\mathcal B}}

\newcommand{\cA}{{\mathcal A}}

\newcommand{\cH}{{\mathcal H}}

\newcommand{\cG}{{\mathcal G}}
\newcommand{\btau}{{\boldsymbol{\tau}}}
\newcommand{\T}{\dagger}

\newcommand{\ttHP}{\mathtt{HP}}

\newcommand{\mat}[4]{\left[ \hspace{-0.4em} \begin{array}{cc}
#1 & #2  \\
#3 & #4  \\
\end{array} \hspace{-0.4em}\right]}

\newcommand{\ind}{\mathbf 1}

\newcommand{\sinc}{\operatorname{sinc}}

\definecolor{darkgreen}{rgb}{.2,0.4,0.3}

\newcommand{\q}{q}

\newcommand{\uu}{\mathfrak{u}}

\newcommand{\tq}{{\breve q}}

\begin{document}
\maketitle

\begin{abstract}
We study the $\Sineb$ process, the bulk point process scaling limit of beta-ensembles. We provide a representation of  its pair  correlation function for all $\beta>0$ via a stochastic differential equation, and also via an infinite system of differential equations. We show that the pair correlation function is continuous in $\beta$, and provide estimates for its asymptotic decay.
When $\beta=2n$ is an even integer, we characterize the pair correlation function through a linear system of $n$ ordinary differential equations and derive the function's power series expansion. Our results recover the classical explicit formula for the pair correlation function in the $\beta=2$ and $4$ cases. 
    We obtain our results by studying the one-point function of the  $\HPb$ process, the point process limit of the circular Jacobi beta-ensemble.

\end{abstract}

\section{Introduction}

The $\Sineb$ process is the bulk scaling limit of beta-ensembles. It is a translation invariant process with density $\tfrac1{2\pi}$. It was first introduced as the bulk scaling limit of the Gaussian beta-ensemble for general $\beta>0$ in \cite{BVBV}. Later it was shown to be the same as the point process limit of the circular beta-ensemble (\cite{BVBV_sbo},\cite{Nakano}) which was first derived by Killip and Stoiciu in \cite{KS}. Finally, \cite{BEY} proved that $\Sineb$ is the universal bulk scaling limit of a general class of beta-ensembles.

For a simple point process $\Xi$ on $\R$ and a Borel set $A\subset \R$, we denote by $\Xi(A)$ the number of points of $\Xi$ in $A$. 
The $k$-point correlation (or joint intensity) function of $\Xi$ (if it exists) is the symmetric function $\rho^{(k)}:\R^k\to [0,\infty)$ that satisfies the following: if  $D_1, \dots, D_k\subset \R$ are disjoint Borel sets then
\begin{align}
E[\prod_{j=1}^k \Xi(D_j)]=\int_{D_1} \cdots  \int_{D_k} {\rho^{(k)}(x_1,\dots,x_k)} dx_k\, \dots dx_1.
\end{align}
The function $\rho^{(1)}$ (if it exists) is just the Radon–Nikodym derivative of the expected counting measure $\nu(A)=E[\Xi(A)]$ with respect to the Lebesgue measure, this is the density of the point process $\Xi$.

Many of the classical results of random matrix theory use $k$-point correlation functions to identify the point process limit of various random matrix ensembles. For the classical values $\beta=1,2,4$, the finite beta-ensembles possess special algebraic structures. These processes are determinantal (for $\beta=2$) or Pfaffian (for $\beta=1, 4$), which means that it is possible to express their correlation functions in terms of determinantal or Pfaffian formulas built from the orthogonal polynomials of the reference measure of the ensemble. For circular and Gaussian beta-ensembles, one can take limits of these formulas to obtain explicit expressions for the $k$-point correlation of the $\Sineb$ process for $\beta=1,2,4$  (see \cite{AGZ, mehta, ForBook}).

In contrast, the $\Sineb$ process           for general $\beta>0$ is  characterized via its counting function using a one-parameter family of stochastic differential equations (\cite{BVBV}, \cite{KS}), or as the spectrum of a random differential operator (\cite{BVBV_sbo}). It has not been clear how one can obtain the classical $\beta=1,2,4$ results from these characterizations of $\Sineb$, and whether it is possible to get any information about the correlation functions. Our paper provides the first positive results in this direction, by studying the pair (or two-point) correlation function $\rho^{(2)}_\beta(x,y)$ of the $\Sineb$ process.

\subsection{Results for general $\beta>0$}

Our first result provides a characterization of $\rho^{(2)}_\beta(0,\lambda)$ in terms of a stochastic differential equation. Since the $\Sineb$ process  is translation and reflection invariant, we have $\rho^{(2)}_\beta(x,y)=\rho^{(2)}_\beta(0,|y-x|)$. 
Here, and in the rest of the paper, for $x\in \R$ and a non-negative integer $k$ we use the following notation for the rising factorial (or Pochhammer symbol):
\begin{align}\label{eq:Poch}
    (x)^{\uparrow k}=\prod_{j=0}^{k-1}(x+j), 
\end{align}
with the empty product defined as 1. 
\begin{theorem}[Pair correlation function for $\Sineb$]\label{thm:sineb_corr}
    Let $\beta>0$, $\lambda>0$, and consider the unique strong solution $\alpha_{\lambda}(u)$ of the following stochastic differential equation on $(-\infty,0]$:
    \begin{align}
    \label{eq:alphaSDE_b}
	d\alpha_\lambda(u) = \lambda\tfrac{\beta}{4}e^{\frac{\beta}{4}u} du -\frac{\beta}{2} \sin(\alpha_\lambda(u)) du+ \Re[(e^{-i\alpha_\lambda(u)}-1)dZ], \qquad \lim_{u\to -\infty} \alpha_\lambda(u)=0. 
\end{align}
    The  pair correlation function of the $\Sineb$ process is given by
\begin{align}\label{eq:twopoint_cor}
        \rho_\beta^{(2)}(0,\lambda)=\frac{1}{4\pi^2}+\frac{1}{2\pi^2}\sum_{k=1}^\infty  \frac{(-\beta/2)^{\uparrow k}}{(1+\beta/2)^{\uparrow k}} E[\cos(k \alpha_\lambda(0))].
    \end{align}  
\end{theorem}
Section \ref{sec:background} and the Appendix provide more detail about the diffusion \eqref{eq:alphaSDE_b}.
For $\delta>0$ one has the uniform bound
\begin{align}
\left|\frac{(-\delta)^{\uparrow k}}{(1+ \delta)^{\uparrow k}}\right|\le c_\delta k^{-1-2  \delta},
   \label{eq:coeff_bnd} 
\end{align}
with a $\delta$-dependent  constant $c_\delta$ that is locally bounded. Hence \eqref{eq:twopoint_cor} gives an absolutely convergent sum.  In fact, the sum can be evaluated to provide the simplified expression
\begin{align}
     \rho_\beta^{(2)}(0,\lambda)=\frac{2^{\beta} \Gamma(1+\beta/2)^2}{4\pi^2 \Gamma(1+\beta)} E\left[|\sin(\alpha_\lambda(0)/2)|^\beta\right],
\end{align}
see Lemma \ref{lem:h_delta} below.

The proof of Theorem \ref{thm:sineb_corr} implies the existence of the pair correlation function $\rho^{(2)}_\beta(x,y)$. In Theorem \ref{thm:corr_existence} in the Appendix, we provide a self-contained proof  for the existence of the $k$-point correlation functions for all $k\ge 1$. The proof relies on the finite exponential moment of the counting function proved in \cite{BVBV}, and the DLR-characterization of the $\Sineb$ process given in \cite{DLR2021}.

Next we provide a characterization of the pair correlation function in terms of an infinite system of differential equations, without the use of the diffusion \eqref{eq:alphaSDE_b}. For $k\ge 1$ consider the differential equation  
\begin{align}\label{eq:EQk}
    \tfrac{\beta}{4} \lambda  q_k'(\lambda)&=\tfrac12 k(k+\beta/2)  q_{k-1}(\lambda)-k^2  q_k(\lambda)+\tfrac12 k(k-\beta/2)  q_{k+1}(\lambda)+i \frac{\beta}{4} k \lambda  q_k(\lambda) 
    \tag{$\mathsf{eq}_k$}
\end{align}
with $\lambda\in(0,\infty)$, $q_k(\lambda)\in \CC$ and $q_{0}(\lambda)=1$.

\begin{theorem}[Analytic characterization of the pair correlation function]\label{thm:ODE_representation} Consider the system of differential equations \eqref{eq:EQk}, $k\ge 1$ with the additional requirements that $\lim\limits_{\lambda\to 0^+} q_k(\lambda)=1$ for each $k\ge 1$ and $\sup\limits_{k\ge 1,\lambda> 0} |q_k(\lambda)|<\infty$.   This system has a unique solution, and with this solution we have
\begin{align}\label{eq:twopoint_cor2}
        \rho_\beta^{(2)}(0,\lambda)=\frac{1}{4\pi^2}+\frac{1}{2\pi^2}\sum_{k=1}^\infty  \frac{(-\beta/2)^{\uparrow k}}{(1+\beta/2)^{\uparrow k}} \Re q_k(\lambda), \qquad \qquad \lambda>0.  
    \end{align}     
    
\end{theorem}

One can realize the solutions of \eqref{eq:alphaSDE_b} for all $\beta>0$ and $\lambda>0$ on the same probability space, in a way that the solution is almost surely continuous separately in its parameters (see Proposition \ref{prop:sde_cont} below). This allows us to prove the following regularity result. 
\begin{theorem}[Regularity of the pair correlation function]\label{thm:beta_cont}
    The function $\rho_{\beta}^{(2)}(0,\lambda)$ is continuous separately in $\beta$ and $\lambda$.  For $\beta>2, \lambda>0$ the function is jointly continuous in $\lambda$ and $\beta$. 
\end{theorem}
Identifying the large $\lambda$ behavior of 
the  pair correlation function $\rho^{(2)}_\beta(0,\lambda)$ for general $\beta>0$ has been a long-standing open problem. Using the representations of Theorems \ref{thm:sineb_corr} and \ref{thm:ODE_representation}, we are able to provide the following bound on the asymptotic decay of the truncated correlation function $\rho^{(2)}_\beta(0,\lambda)-\rho^{(1)}_\beta(\lambda)^2=\rho^{(2)}_\beta(0,\lambda)-\frac{1}{4\pi^2}$. 
\begin{theorem}[Decay of the truncated pair correlation]\label{thm:large_lambda}
Fix $\beta>0$. There is a $\beta$-dependent constant $c_\beta$ so that we have the following bound for $\lambda>0$: 
\begin{align}\label{eq:large_lambda}
     \left|\rho^{(2)}_\beta(0,\lambda)-\frac{1}{4\pi^2}\right|\le c_\beta\left(\lambda^{-\beta/2}+\lambda^{-4/\beta}\right).
\end{align} 
\end{theorem}
For $\beta\ge 2\sqrt{2}$, the main term in our upper bound is $\lambda^{-4/\beta}$, which matches the conjectured order of the truncated pair correlation function in this regime (see equation \eqref{eq:rho2_2n_asympt} below, and the discussion around it). While originally working on this manuscript, we have learned from Laure Dumaz and Martin Malvy \cite{DM2026} that they have also obtained polynomially decaying bounds on the truncated pair correlation function of $\Sineb$, with methods that are different from ours. 

Our bounds are not optimal in the small $\beta$ regime, and they do not describe the conjectured oscillating behavior of the truncated correlation function. A more careful analysis of the system \eqref{eq:EQk}, $k\ge 1$, may yield sharper bounds. Building on our representation for $\rho_\beta^{(2)}$, Assiotis \cite{Theo2026} recently carried out such an analysis, obtaining sharper decay estimates and identifying the leading term in $\rho_\beta^{(2)}(0,\lambda)-\tfrac1{4\pi^2}$ as $\lambda\to \infty$.

\subsection{Results for $\beta=2n$}

The beta-ensembles do not possess determinantal/Pfaffian structure for values other than the classical values 1, 2, and 4. However, using Jack polynomials and generalized hypergeometric functions, Forrester was able to derive exact formulas for  the circular beta-ensemble for $\beta=2n$, $n\in \ZZ_+$ (see \cite{Forrester_1992}, \cite{Forrester_1994}, and Chapter 13 of \cite{ForBook}). These formulas lead to a $2n$-variable integral representation of the pair correlation function of the $\Sineb$ process when $\beta=2n$. We will review these results in more detail in the \textit{Historical background} section below.

Our representation in Theorem \ref{thm:sineb_corr} leads to   new results for $\beta=2n$.  
The Pochhammer symbol $(-x)^{\uparrow k}$ is zero if $x$ is a positive integer and $k>x$. This means that the expressions \eqref{eq:twopoint_cor} and \eqref{eq:twopoint_cor2} in Theorems \ref{thm:sineb_corr} and \ref{thm:ODE_representation} become  finite sums if $\beta=2n$. This allows us to express $\rho^{(2)}_{2n}(0,\lambda)$ in terms of  the solution of an $n$-dimensional linear ODE system, and also as a power series with coefficients obtained from a matrix-valued recursion. 
 For the full result, see Theorem \ref{thm:even_beta}, Corollary \ref{cor:int_case_rho}, and the discussion following these statements. Here, we only present in detail the result of the power series representation.

Fix $n\in \Z_+$.  
Let $\mathbf{A}_{n}$ be an $n\times n$ tridiagonal  and $\mathbf{B}_n$ an $n\times n$ diagonal matrix with nonzero entries given by 
\begin{align}\label{eq:AB}
      [\mathbf{A}_n]_{k,k}=-k^2, \quad [\mathbf{A}_n]_{k,k-1}=\frac12 k(k+n), \quad [\mathbf{A}_n]_{k,k+1}=\frac12 k(k-n),\quad [\mathbf{B}_n]_{k,k}=k.
\end{align}
Let $\mathbf{e}_n, \mathbf{v}_n\in \R^n$ be defined as 
\begin{align}\label{eq:vectors}\mathbf{e}_n=[1,0,\dots,0]^\dag, \qquad 
[\mathbf{v}_n]_k=(-1)^k\frac{ \binom{2n}{n+k}}{\binom{2n}{n}}, \qquad 1\le k\le n,
\end{align}
where $\mathbf{x}^\dag$ denotes the transpose of $\mathbf{x}$. 

\begin{theorem}[Power series representation for $\rho_{2n}^{(2)}(0,\lambda)$]\label{thm:sineb_series}
Define $\mathbf{s}_k \in \CC^n$ recursively as
\begin{align}\label{eq:q_coeff}
    \mathbf{s}_0&=-\frac{n+1}{2}\mathbf{A}_n^{-1}\mathbf{e}_n, \qquad \mathbf{s}_k=i \left(k \mathbf{I}-\tfrac{2}{n} \mathbf{A}_n\right)^{-1} \mathbf{B}_n \mathbf{s}_{k-1}, \qquad k\ge 1.
\end{align}
Then 
\begin{align}\label{eq:rho2_2n}
    \rho_{2n}^{(2)}(0,\lambda)=\frac{1}{2\pi^2} \sum_{j=1}^\infty \mathbf{v}_n^{T}  \,\mathbf{s}_{2j}\,  \lambda^{2j}.
\end{align}    
\end{theorem}
As part of the proof, we will show that the sequence $\mathbf{s}_k, k\ge 0$ is well-defined, and the power series converges for all $\lambda\in \R$. 

Theorem \ref{thm:sineb_series} allows us to recover the well-known Sine-kernel for the $\beta=2$ case. Indeed, if  $\beta=2$ (i.e.~$n=1$), then we have $\mathbf{A}_1=-1$, $\mathbf{B}_1=1$, $\mathbf{v}_1=-\tfrac{1}{2}$,  $\mathbf{s}_k=\tfrac2{(k+2)!} i^k $, hence \eqref{eq:rho2_2n} yields
\begin{align}\label{eq:sine_kernel}
    \rho_{2}^{(2)}(0,\lambda)=\frac{1}{4\pi^2}\left(1-\frac{\sin^2(\tfrac{\lambda}{2})}{(\lambda/2)^2}\right).
\end{align}
In  Section \ref{subs:explicit}, we provide a closed formula for the function $\sum_{j=0}^\infty \mathbf{s}_{j}  \lambda^{j}$ in the $\beta=4$ case, which allows us to recover the known expression  for the pair correlation function of the $\Sine_4$ process (see \eqref{eq:beta=4} below). In Section \ref{subs:ODE_repr}, we will outline a method that produces an order $2n$ linear differential equation that is solved by $\rho_{2n}^{(2)}(0,\lambda)$.

Using computer algebra and Theorem \ref{thm:sineb_series}, it is easy to identify specific coefficients  in the asymptotic expansion of $\rho_{2n}^{(2)}(0,\lambda)$ near 0 if $n$ is not too large. For example, here are the first couple of non-zero terms for the series expansion of the pair correlation function in the case of $\beta=6$:
\begin{align*}
    \rho_{6}^{(2)}(0,\lambda)=\frac{27  \lambda ^6}{179200 \pi ^2}-\frac{243 \lambda ^8}{19712000 \pi ^2}+\frac{243 \lambda ^{10}}{482944000 \pi ^2}-\frac{729 \lambda ^{12}}{53365312000 \pi ^2}+\mathcal{O}(\lambda^{14}).
\end{align*}
As an additional application, we can show that when $\beta=2n$,  the first nonzero term in our expansion  is of order $\lambda^{2n}$, and identify its coefficient. This result also follows directly from the integral representation of Forrester, see the discussion below.

\begin{corollary}\label{cor:small_lambda}
  We have the following asymptotics for $\lambda$ near 0:  
  \begin{align}\label{eq:small_lambda}
      \rho_{2n}^{(2)}(0,\lambda)= \frac1{4\pi^2} \cdot \frac{n^{2n} (n!)^3 }{(2n)!(3n)! }\lambda^{2n}+\mathcal{O}(\lambda^{2n+2}).
  \end{align}
  
\end{corollary}

\subsection{Strategy of the proof}

Our starting point is the following observation from \cite{BVBV_Palm}: if one can properly define the $\Sineb$ process conditioned to have a point at 0, then the one-point function of the conditioned process is  $\frac{\rho_\beta^{(2)}(0,x)}{\rho_\beta^{(1)}(0)}$. 

The $\Sineb$ process is the point process limit of the circular beta-ensemble. We can readily condition a size  $n$ circular beta-ensemble to have a point at angle 0. After removing this point, we obtain the size $n-1$ circular Jacobi beta-ensemble with parameter $\delta=\beta/2$.  The scaling limit of the circular Jacobi beta-ensemble was first derived in \cite{AN2021}. Later in \cite{LV_HP} various characterizations were given for the limiting point process, which was named $\HPb$  (after Hua and Pickrell, who studied  the circular Jacobi beta-ensemble  for $\beta=2$ in \cite{Hua} and  \cite{Pickrell}). See Section \ref{sec:background} for precise statements and definitions. Using the results of \cite{BVBV_Palm}, we will show that the $\HPb$ process with $\delta=\beta/2$ is the same as the $\Sineb$ process conditioned to have a point at 0, with that point removed (see Proposition \ref{prop:Palm} below). In particular, with $\rho_{\beta,\delta}^{(1)}(\lambda)$ denoting the density of $\HPb$, we have
\begin{align}\label{eq:palm}
  \rho_\beta^{(2)}(0,\lambda)=  \rho_{\beta,\beta/2}^{(1)}(\lambda){\rho_\beta^{(1)}(0)}=\frac{1}{2\pi} \rho_{\beta,\beta/2}^{(1)}(\lambda).
\end{align}
\cite{LV_HP} provides a characterization of $\HPb$   as the result of a shooting problem involving an SDE similar to \eqref{eq:alphaSDE_b}. The shooting problem characterization allows us to compute the one-point function  $\rho_{\beta,\delta}^{(1)}(\lambda)$ of $\HPb$ for $\beta>0, \delta>0$ in terms of certain expectations related to the solution of the SDE.  This   leads to the proof of  Theorem \ref{thm:sineb_corr}, which is the starting point of all of our other results.

\subsection{Historical background}

For $\beta=2$, the circular and the circular Jacobi beta-ensembles are determinantal, the $k$-point correlation functions can be obtained as the determinant of a $k\times k$ matrix built from a kernel function. Taking the scaling limit of these formulas, one obtains exact determinantal expressions for all correlation functions of $\Sineb$ and $\HPb$ for $\beta=2$, in particular, one obtains \eqref{eq:sine_kernel}. For $\beta=1$ and $\beta=4$, one can write down the correlation functions of the finite and the limiting ensembles in terms of Pfaffian formulas (\cite{mehta}, \cite{AGZ}, \cite{ForBook}). For $\beta=4,$ this leads to the following expression for the pair correlation function of the $\Sineb$ process:
\begin{align}\label{eq:beta=4}
     &\rho_{4}^{(2)}(0,\lambda) = \frac{1}{4\pi^2}\left(1- \sinc^2(\lambda) + \sinc'(\lambda)\int_0^{\lambda} \sinc(t)dt\right),\qquad \sinc(x)=\frac{\sin(x)}{x}.
\end{align}
 When $\beta=2n$ (i.e.~when $\beta$ is an even integer), Forrester 
  provided a representation for the correlation functions of the circular beta-ensemble in terms of generalized hypergeometric functions and Jack polynomials (see \cite{Forrester_1992}, \cite{Forrester_1994}, and Chapter 13 of \cite{ForBook}), and evaluated the scaling limit of these functions. 
The recent work of Assiotis and Najnudel \cite{AN2026} implies that these scaling limits agree with the  correlation functions of the limit point process $\Sineb$. (See below for further discussion on \cite{AN2026}.)
Forrester's formula for the pair correlation function then gives
\begin{align}\notag
 \rho_{2n}^{(2)}(0,\lambda)=&\frac1{4\pi^2} \cdot \frac{n^{2n} (n!)^3 }{(2n)!(3n)! } \frac{e^{-i n \lambda} \lambda^{2n}}{S_{2n}(-1+1/n,-1+1/n, 1/n)}\times\\
 &\qquad \int_{[0,1]^{2n}} \prod_{j=1}^{2n} \left(e^{i \lambda u_j} u_j^{-1+1/n} (1-u_j)^{-1+1/n} \right)\prod_{j<k}|u_j-u_k|^{2/n} \prod_{j=1}^{2n}du_j .\label{eq:paircorr_For}
\end{align}
Here $S_{2n}(-1+1/n,-1+1/n, 1/n)$ is the Selberg integral on $[0,1]^{2n}$ that one obtains by setting $\lambda=0$ in the integral in the second line of \eqref{eq:paircorr_For}. By taking $\lambda\to 0$, we can recover the small $\lambda$ asymptotics of \eqref{eq:small_lambda}, but it is  not clear why the integral representation \eqref{eq:paircorr_For} is equivalent to 
the power series representation  \eqref{eq:rho2_2n}. The recent preprint \cite{WF2026} provides a direct proof of this fact.

By analyzing the integral in \eqref{eq:paircorr_For},
Forrester provides the following large $\lambda$ asymptotics of \eqref{eq:paircorr_For} for $\beta=2n$:
\begin{align}\notag
4\pi^2    \rho_{\beta}^{(2)}(0,\lambda)=&1-\frac{4}{\beta \lambda^2}+\mathcal{O}(\lambda^{-4})\\
&+\sum_{k=1}^{n}\frac{a_k}{\lambda^{4k^2/\beta}}\left(\cos(k \lambda)+b_k\frac{\sin(k\lambda)}{\lambda}+\mathcal{O}\left(\frac{|\cos(k\lambda)|}{\lambda^2}\right)\right)\label{eq:rho2_2n_asympt},
\end{align}
with explicitly given coefficients $a_k, b_k$ (see Proposition 13.13, \cite{ForBook}).
Similar asymptotics, with an infinite sum of oscillating terms in \eqref{eq:rho2_2n_asympt}, 
were obtained by Haldane \cite{Haldane1981}
in the study of one-component, one-dimensional quantum fluids. It is conjectured that Haldane's asymptotics hold 
for the pair correlation $\rho_{\beta}^{(2)}(0,\lambda)$ for general $\beta>0$ as well (see e.g.~\cite{Forrester_1984}). Note that the largest order term in this conjectured expansion is of order $\lambda^{-\min(2,4/\beta)}$. Our bound in Theorem \ref{thm:large_lambda} matches this order when $\beta\ge 2\sqrt{2}$.

In \cite{Forrester_20212} (building on the results of \cite{FR2012}) Forrester outlines how one can represent $\rho_{2n}^{(2)}(0,\lambda)$ as the solution of an order $2n+1$ differential equation, and provides these representations for $\beta=1,2,4,6$. \cite{Forrester_2021} gives similar results for the scaling limit of the one-point function of the circular Jacobi beta-ensemble.   \cite{FS2025} gives a precise description on how the $k$-point correlation functions for the finite circular beta-ensembles converge to their limit in the $\beta=1,2,4$, and $\beta=2n$ cases, providing an asymptotic expansion. \cite{FW2021small} proves that the smallest gaps in the circular beta-ensemble, after an appropriate normalization, converge to a certain Poisson process when $\beta=n$ is a positive integer.

The random operator approach to study the point process limits of beta-ensembles was initiated by Dumitriu-Edelman \cite{DE} and Edelman-Sutton \cite{ES}. Building on their work, the point process limits of beta-ensembles have been characterized as the spectrum of random differential operators, or by describing the counting function using systems of stochastic differential equations \cite{RRV}, \cite{RR}, \cite{KS}, \cite{BVBV}, \cite{BVBV_sbo}, \cite{LV_HP}. These descriptions provided tools to study various asymptotic quantities related to these processes (e.g.~large gap asymptotics \cite{BVBV2}, \cite{RRZ}, asymptotics for the counting function \cite{KVV}, \cite{HP2018}, large deviations for the number of points in a small or a large interval \cite{HV}, \cite{HV2}), but before the current paper there have been only a few results that recovered the non-asymptotic exact formulas known for the classical beta values. The most significant result where this is actually achieved  is from \cite{Spike1}, where  the authors recovered explicit formulas for $\beta=2,4$ for the soft-edge limit of the Gaussian beta-ensemble, providing an independent proof of the Painlev\'e representations of 
Tracy-Widom distributions originally given in \cite{TW1994}, \cite{TW1996}.

There are several other promising approaches to study point process limits of beta-ensembles for general $\beta>0$. \cite{DLR2021} gives a characterization of the $\Sineb$ process using the Dobrushin-Lanford-Ruelle formalism, \cite{Leble2021} uses these results to prove a CLT for the linear statistics for $\Sineb$. 
The method of loop equations (introduced in \cite{Johansson1998}) provides a powerful tool to study finite beta-ensembles, see e.g.~\cite{Lambert2021} for an application.

We also mention here the unpublished preprint  \cite{okounkov1997}, where the limit of the $k$-point correlation function of the circular beta-ensemble is presented in the case when $\beta$ is a rational number.

\subsection{Recent results}

Since the release of the first version of our paper, several significant results related to the correlation functions of the $\Sineb$ process have appeared. We attempt to briefly summarize these new developments in this section.

In \cite{AN2026} Assiotis and Najnudel derive an expression for the $k$-point correlation function of the $\Sineb$ process in terms of an expectation involving a random analytic function related to the $\HPb$ process with $\delta=k \beta/2$ (see Theorem 1.8). As a corollary, they are able to identify the asymptotics of the $k$-point correlation function when the $k$-points are all close to a given value $x$. For $k=2$ this extends our Corollary \ref{cor:small_lambda} for general $\beta>0$. Dumaz and Malvy in \cite{DM2026}  provide  decay estimates for  correlation functions of {all orders} for $\Sineb$ in the long-range regime. The  paper of Bourgade and Huang \cite{BH_2026}  proves a new characterization of $\Sineb$ for rational $\beta$ values via loop equation hierarchies. 

In \cite{betasix2026} the authors of this paper (together with additional co-authors) derive an explicit formula for the pair correlation function of the $\Sineb$ process for $\beta=6$ in terms of Bessel functions of the first kind. This is the first single-variable special function representation of the pair correlation function $\rho_{\beta}^{(2)}(0,\lambda)$ that goes beyond the classical $\beta=1,2,$ and 4 cases. The proof relies on the results of the current paper (more specifically Theorem \ref{thm:sineb_series}, Theorem \ref{thm:even_beta}, and Corollary \ref{cor:int_case_rho}). 

The manuscript \cite{WF2026} resolves  two open problems stated in the first version of this paper. The author extends Theorem \ref{thm:beta_cont} to show that $\rho_{\beta}^{(2)}(0,\lambda)$ is jointly continuous in $\lambda$ and $\beta$ (without the $\beta>2$ restriction), and shows a direct equivalence between the representation given by Theorem \ref{thm:even_beta} and the integral formula \eqref{eq:paircorr_For} due to Forrester for $\rho_{2n}^{(2)}(0,\lambda)$. 

Finally, Assiotis in \cite{Theo2026} identifies the leading term of $\rho_{\beta}^{(2)}(0,\lambda)$ as $\lambda\to \infty$ using the representation of the current paper, strengthening the results of our Theorem \ref{thm:large_lambda}.\medskip

The first version of this manuscript ended with a list of open problems. Due to the rapid developments summarized above, most of these problems have since been resolved. We now highlight two questions that remain open.

It would be interesting to understand whether our representations of the pair correlation function can be connected directly to the results of Okounkov in \cite{okounkov1997}.
\begin{problem}
Show directly that the representations of Theorems \ref{thm:sineb_corr} and  \ref{thm:ODE_representation} lead to the pair correlation formula given in \cite{okounkov1997} for rational $\beta$ values.
\end{problem}
Our proof of Theorem \ref{thm:1p_HP} relies on the assumption that $\delta$ is a positive real number. It is natural to ask whether the result can be extended to a broader range of values of $\delta$.
\begin{problem} Extend Theorem \ref{thm:1p_HP} to express the one-point function of $\HPb$ for complex values of $\delta$ satisfying $\Re \delta>-1/2$. \end{problem}

\subsection{Outline of the paper}

In Section \ref{sec:background}, we review some known results on the circular and circular Jacobi beta-ensembles and their point process limits, present some properties of the $\alpha_\lambda$ diffusion of \eqref{eq:alphaSDE_b} (in a more general setting), and
present the proof of equation \eqref{eq:palm}. Section \ref{sec:1point} contains the proof of Theorem \ref{thm:sineb_corr} and the analogous result on the density of the $\HPb$ process with $\delta>0$.  
Section \ref{sec:analytic} presents the proofs of our results about the analytic properties of $\rho_{\beta}^{(2)}(0,\lambda)$ for general $\beta>0$ (Theorems \ref{thm:ODE_representation}, \ref{thm:beta_cont}, and \ref{thm:large_lambda}). Section \ref{sec:delta_n} collects our proofs of our results for the $\beta=2n$ case, with corresponding versions for the $\HPb$ process for $\delta=n$. We collect some more technical proofs in the Appendix (Section \ref{sec:Appendix}).

\subsection*{Acknowledgments}

We are grateful to the referees for their careful reading of our manuscript. Their insightful comments and suggestions led to substantial improvements in the presentation of the paper.
The authors  thank Laure Dumaz, Martin Malvy and Yun Li for useful discussions, and Peter Forrester and Theodoros Assiotis for additional comments and references. 
B.V.~was partially supported by  the University of Wisconsin – Madison Office of the Vice Chancellor for Research and Graduate Education with funding from the Wisconsin Alumni Research Foundation and by the National Science Foundation award DMS-2246435.

\subsection*{AI statement}

 There were no AI tools used in the first version of this manuscript. During the revision process, the authors used AI tools to check mathematical arguments and help identify possible errors.  The authors carefully reviewed all AI-generated feedback and take full responsibility for the content of the manuscript. In particular, all proofs were produced by the authors.

\section{Preliminaries on beta-ensembles}\label{sec:background}

We review some useful tools and background material concerning the $\Sineb$ process and its generalization, the $\HPb$ process. Our presentation emphasizes the main tools needed in our proofs,  additional details are provided in the   Appendix (Section \ref{sec:Appendix}).

\subsection{The circular beta-ensemble, the circular Jacobi beta-ensemble, and their limits}

Let $\beta>0$, $n\in \Z_+$, and $\delta\in \CC$ with $\Re \delta>-1/2$. The size $n$ circular Jacobi beta-ensemble with parameters $\beta, \delta$ is the joint distribution  of 
$n$  points $\{e^{i \theta_1}, \dots, e^{i \theta_n}\}$ on the unit circle $\{|z|=1\}$, where the joint density function of the angles $\theta_j\in [-\pi,\pi)$, is given by
\begin{align}
 \frac{1}{Z_{n,\beta,\delta}^{\textup{CJ}}} \prod_{1\le j<k\le n} \left|e^{i \theta_j}-e^{i \theta_k}\right|^\beta  \prod_{k=1}^n (1-e^{- i \theta_k})^{ \delta}(1-e^{ i \theta_k})^{ \bar\delta} \label{eq:PDF_cjacobi}, \qquad \theta_j\in [-\pi,\pi).
\end{align}
Here  $Z_{n,\beta,\delta}^{\textup{CJ}}$ is an explicitly computable normalizing constant. For $\delta=0$ the distribution becomes the circular beta-ensemble. We  write $\Lambda_n\sim \textup{CJ}_{n,\beta,\delta}$ to denote that the random set $\Lambda_n=\{\theta_1, \dots, \theta_n\}$ has  joint density given by \eqref{eq:PDF_cjacobi}, and we use the notation $\textup{C}_{n,\beta}$ for $\textup{CJ}_{n,\beta,0}$, the size $n$ circular beta-ensemble.

Killip and Stoiciu \cite{KS} showed that if $\Lambda_n\sim \textup{C}_{n,\beta}$ then $n\Lambda_n$ converges to a point process as $n\to \infty$. The point process limit was characterized via a one-parameter family of stochastic differential equations via a shooting problem.  This process was later shown to be the same as the   $\Sineb$ process, the bulk scaling limit of the Gaussian beta-ensemble \cite{BVBV_sbo}, \cite{Nakano}, which can also be characterized as the spectrum of a random differential operator.

In \cite{LV_HP} it was shown that if $\Lambda_n\sim \textup{CJ}_{n,\beta,\delta}$ then $n \Lambda_n$ converges to a point process $\HPb$. (Note that $\HPb$ with $\delta=0$ is just $\Sineb$.) This process can also be characterized in multiple ways (as the spectrum of a certain random differential operator, as the zero set of a certain random entire function). We review the characterization that will be useful for our purposes.  In the Appendix we provide a more detailed background. 
 Although the circular Jacobi beta-ensemble is defined for parameters $\beta>0, \Re \delta>-1/2$, from this point we will assume that $\delta$ is real and non-negative. We do not use the general complex $\delta$ parameters in our proofs, and making $\delta$ real will  simplify some of the notation. 

Let $Z=B_1+i B_2$ be a two-sided complex Brownian motion with standard independent real and imaginary parts. 
For a given $\beta>0, \delta\ge 0$  consider the unique strong solution $\alpha_\lambda(u)$  of the following one-parameter family of stochastic differential equations:  
\begin{align}
    \label{eq:alphaSDE}
	d\alpha_\lambda(u) = \lambda\tfrac{\beta}{4}e^{\frac{\beta}{4}u} du -\delta  \sin(\alpha_\lambda(u)) du+ \Re[(e^{-i\alpha_\lambda(u)}-1)dZ],\qquad \lambda\in \R, u\in(-\infty,\infty) 
\end{align}
with  initial condition
\begin{align}\label{eq:SDE_initial}
    \lim_{u\to-\infty} \alpha_\lambda(u)=0, \qquad \lambda\in \R. 
\end{align}
(Proposition \ref{prop:alphaSDE} of the next section provides more detail about this process.)
Let $\Theta$ be a random variable on $[0,2\pi)$ that is independent of $Z$ and has probability density function  
\begin{align}\label{eq:Theta}
 h_\delta(\theta)= {\frac{1}{2\pi}} \frac{\Gamma(1+\delta)^2}{\Gamma(1+2\delta)}|1-e^{-i \theta}|^{2\delta}  , \qquad \theta\in [0,2\pi).
\end{align}
Theorem 20 of \cite{LV_HP} proved that
\begin{align}\label{eq:HP}
    \HPb\ed\{\lambda\in \R: \alpha_\lambda(0)=\Theta \mod 2\pi\}.
\end{align}
For $\delta=0$ this recovers the characterization of $\Sineb$ given in \cite{KS} by Killip and Stoiciu,  in that case $\Theta$ is uniform on $[0,2\pi)$. Note that the diffusion \eqref{eq:alphaSDE_b} is the same as  \eqref{eq:alphaSDE} with initial condition \eqref{eq:SDE_initial} and $\delta=\beta/2$.

\subsection{Properties of the $\alpha_\lambda$ diffusion}

The following proposition summarizes some of the useful properties of the diffusion \eqref{eq:alphaSDE} proved in \cite{LV_HP},  see the Appendix for additional details.

\begin{proposition}\label{prop:alphaSDE}Fix $\beta>0$ and $\delta\ge 0$.

\begin{enumerate}
    \item For any $\nu\in \R$ the system \eqref{eq:alphaSDE} has a unique strong solution $\alpha_{\lambda,\nu}(u)$ on $[\nu,\infty)$ with initial condition $\alpha_{\lambda,\nu}(\nu)=0$. For any given $\nu< u$, with probability one the function $\lambda\mapsto \alpha_{\lambda,\nu}(u)$ is analytic and strictly increasing with $\alpha_{0,\nu}(u)=0$

    \item There is a unique strong solution $\alpha_\lambda(u)$ of \eqref{eq:alphaSDE} on $(-\infty,\infty)$ satisfying the initial condition \eqref{eq:SDE_initial}, which  can be obtained as the a.s.~limit $\alpha_\lambda(u)=\lim\limits_{\nu\to -\infty} \alpha_{\lambda,\nu}(u)$. The limit is monotone increasing for $\lambda>0$ and monotone decreasing for $\lambda<0$. 
    
    For any $u\in \R$ with probability one the function $\lambda\mapsto \alpha_\lambda(u)$ is analytic and strictly increasing in $\lambda$ with $\alpha_0(u)=0$. 

    Moreover, $\alpha_\lambda(u)$ satisfies the following scale-invariance property: for any fixed $s\in \R$ 
\begin{align}\label{eq:scale_inv}
    \left(\alpha_\lambda(u), u\in (-\infty,\infty)\right) \ed  \left(\alpha_{e^{\beta s/4}\lambda}(u-s), u\in (-\infty,\infty)\right). 
\end{align}  
\end{enumerate}
\end{proposition}

For  $\beta>0, \delta\ge 0$, and $k\in \Z_{>0}$ let
\begin{align}\label{def:tildeqk}
    \tilde q_{k,\beta,\delta}(u,\lambda)&=E\left[\exp(i k \alpha_\lambda(u))\right], \qquad u\in \R, \lambda\in \R\\
    q_{k,\beta,\delta}(\lambda)&=\tilde q_{k,\beta,\delta}(0,\lambda)=E\left[\exp(i k \alpha_\lambda(0))\right], \qquad \lambda\in \R.\label{def:qk}
\end{align}
We can extend this definition to include $k=0$ with
\begin{align}
    \tilde q_{0,\beta,\delta}(u,\lambda)=1, \qquad q_{0,\beta,\delta}(\lambda)=1.
\end{align}
We will often drop $\beta, \delta$ from the notation.  
The following lemma collects some properties of the functions $\tilde q_k, q_k$ that we will use repeatedly. 
\begin{lemma}\label{lem:qk}
  For $k\ge 1$ and $\lambda\neq 0$ we have 
    \begin{align}\label{eq:tqk_ODE}
        \partial_u \tilde q_{k}(u,\lambda)&=\tfrac12 k(k+\delta) \tilde q_{k-1}(u,\lambda)+(i k \lambda \tfrac{\beta}{4}e^{\frac{\beta}{4}u}-k^2) \tilde q_k(u,\lambda)+\tfrac12 k(k-\delta) \tilde q_{k+1}(u,\lambda),\\
        \tfrac{\beta}{4} \lambda q'_{k}(\lambda)&=\tfrac{1}{2} k(k+\delta) q_{k-1}(\lambda)+(i k \lambda \tfrac{\beta}{4}-k^2)q_{k}(\lambda)+\tfrac{1}{2} k (k-\delta )q_{k+1}(\lambda).\label{eq:qk_ODE}
    \end{align}
    Moreover, $\tilde q_k(u,\lambda)=q_k(e^{\beta u/4} \lambda)$, $q_k(\lambda)=\bar q_k(-\lambda)$, and
    \begin{align}\label{eq:qk0}
        \lim_{\lambda\to 0} q_k(\lambda)=1.
    \end{align}
\end{lemma}
\begin{proof}
   Fix $k\ge 1$. Applying It\^{o}'s formula to the diffusion  \eqref{eq:alphaSDE} leads to 
    \begin{align}\notag
 d e^{i k\alpha_\lambda(u)}&=\tfrac12 k(k+\delta) e^{i (k-1)\alpha_\lambda(u)} du+(i k \lambda \tfrac{\beta}{4}e^{\frac{\beta}{4}u}-k^2) e^{i k \alpha_\lambda(u)} du+\tfrac12 k(k-\delta) e^{i (k+1)\alpha_\lambda(u)}du  \\
 &\qquad -\tfrac{1}{2} i  k (e^{i \alpha_\lambda(u)}-1) e^{i(k-1)\alpha_\lambda(u)} dZ+\tfrac{1}{2} i  k (e^{i \alpha_\lambda(u)}-1) e^{i k \alpha_\lambda(u)} d\bar Z. \label{eq:QkSDE}
\end{align}
Since the noise terms have bounded coefficients in $u$ (for a fixed $k$), we can take expectations in the integrated version of this SDE on any bounded interval $[u_0,u_1]$.  This yields the integrated version of the equation \eqref{eq:tqk_ODE} on any bounded interval. Since $\tilde q_k$ is bounded by one, we also obtain the equation \eqref{eq:tqk_ODE}.  The scale-invariance identity \eqref{eq:scale_inv} implies $\tilde q_k(u,\lambda)=\tilde q_k(0, e^{\beta u/4} \lambda)=q_k(e^{\beta u/4} \lambda)$. From this it follows that $q_k$ is differentiable on $\R\setminus\{0\}$, and its derivative satisfies the equation \eqref{eq:qk_ODE}. 

The identity $q_k(\lambda)=\bar q_k(-\lambda)$ follows from Proposition \ref{prop:alphaSDE} and the fact that $\alpha_\lambda(u)$ satisfies the same SDE as $-\alpha_{-\lambda}(u)$ (with the driving noise $-d\bar Z$ instead of $dZ$). Finally, \eqref{eq:qk0} is a consequence of the dominated convergence theorem, since  $\alpha_\lambda(0)$ is continuous in $\lambda$ and $\alpha_0(0)=0$.
\end{proof}

\subsection{Palm measure of the $\Sineb$ process}

The Palm measure of a stationary point process $\Xi$ on $\R$ is the process viewed from a `typical' point, with that point shifted to 0. The reduced Palm measure is the resulting process with 0 removed. 
For a precise definition, see Section 13 of \cite{DVJ2}  or \cite{BVBV_Palm}. 
Section 13 of \cite{DVJ2}  shows that for a stationary process $\Xi$ with a finite pair correlation function $\rho_{\Xi}^{(2)}(0,x)$, the density  of the (reduced) Palm measure exists and it is given by 
\[\rho_{\text{Palm}, \Xi}^{(1)}(x)=\frac{\rho_{\Xi}^{(2)}(0,x)}{\rho_{\Xi}^{(1)}(0)}.
\]

\begin{proposition}\label{prop:Palm}
Let $\beta>0$.   The reduced Palm measure of $\Sineb$ is given by $\HPbb$.
\end{proposition}

The proof of Proposition \ref{prop:Palm} follows from  the results of \cite{BVBV_Palm} and \cite{LV_HP}, it is given in the Appendix, see Section \ref{subs:PalmSine}.

Proposition \ref{prop:Palm} together with the discussion before it leads to the following statement.
\begin{proposition}\label{prop:sine_HP}
    Fix $\beta>0$ and let $\rho_{\beta,\beta/2}^{(1)}$ be the density (one-point function) of the  $\HPbb$ process. Let $\rho_{\beta}^{(2)}$ denote the pair correlation function of the  $\Sineb$ process. Then
    \begin{align}
        \rho_{\beta}^{(2)}(0,\lambda)=\frac{1}{2\pi} \rho_{\beta,\beta/2}^{(1)}(\lambda).\label{eq:sine_HP}
    \end{align}
\end{proposition}

\section{The one-point function of $\HPb$}
\label{sec:1point}

The main result of this section is the following theorem. 

\begin{theorem}[One-point function of $\HPb$]\label{thm:1p_HP}
 Let $\delta> 0$ and $\beta>0$. 
Consider the strong solution $\alpha_\lambda(u)$ of \eqref{eq:alphaSDE} with initial condition \eqref{eq:SDE_initial}. Then the one-point function of the $\HPb$ process is 
    \begin{align}\label{eq:rho1_sum}
        \rho^{(1)}_{\beta,\delta}(\lambda)=\frac{1}{2\pi}+\frac{1}{\pi}\sum_{k=1}^\infty  \frac{(-\delta)^{\uparrow k}}{(1+\delta)^{\uparrow k}} E[\cos(k \alpha_\lambda(0))]. 
    \end{align}
\end{theorem}
By the bound \eqref{eq:coeff_bnd} the sum on the right side of \eqref{eq:rho1_sum} is absolutely convergent.

\begin{proof}[Proof of Theorem \ref{thm:sineb_corr}]
Using Theorem \ref{thm:1p_HP} with $\delta=\beta/2$ together with the identity \eqref{eq:sine_HP} of Proposition \ref{prop:sine_HP} yields the statement of Theorem \ref{thm:sineb_corr}.
\end{proof}

In order to prove Theorem \ref{thm:1p_HP}, we first average out  $\Theta$ in the characterization \eqref{eq:HP} of the $\HPb$ process. The following lemma will be useful in this step.

\begin{lemma}\label{lem:h_delta}
    Let $\delta> 0$ and consider the probability density function $h_\delta(\theta)$ defined in \eqref{eq:Theta}. Then 
    \begin{align}\label{eq:hdelta_Fourier}
        h_\delta(\theta)=\frac{1}{2\pi}+\frac1{2\pi}\sum_{k=1}^\infty \frac{ (- \delta)^{\uparrow k} }{(1+\delta)^{\uparrow k}} \left(e^{i k \theta}+e^{-i k\theta}\right). 
    \end{align}   
\end{lemma}
\begin{proof}
The function $h_\delta:[0,2\pi)\to [0,\infty)$ is bounded, hence it has a Fourier expansion in $L^2$. We compute the coefficients, and show pointwise convergence of the resulting series. 

We have
\begin{align*}
    (1-z)^\delta=\sum_{k=0}^\infty \frac{(-\delta)^{\uparrow k}}{k!} z^k,
\end{align*}
where this series is absolutely convergent for $|z|=1$. This leads to 
\begin{align*}
     (1-e^{i \theta})^{\delta}(1-e^{-i \theta})^{ \delta}&= \sum_{m, n\ge 0}
\frac{(- \delta)^{\uparrow m} (-\delta)^{\uparrow n}}{m! n!} e^{i \theta (m-n)},
\end{align*}
where the sum is still absolutely convergent. For $k\ge 0$ the coefficient of $e^{i k \theta}$ is given by
\begin{align*}
    \sum_{n=0}^\infty \frac{(-\delta)^{\uparrow n+k} (-\delta)^{\uparrow n}}{(n+k)! n!}&=
    (- \delta)^{\uparrow k}\sum_{n=0}^\infty  \frac{(- \delta+k)^{\uparrow n} (-\delta)^{\uparrow n}}{(n+k)! n!}\\
    &=(-\delta)^{\uparrow k} \frac{1}{k!} { }_2F_1(-\delta+k, -\delta, k+1,1)\\
    &=(- \delta)^{\uparrow k}  \frac{ \Gamma(1+2\delta)}{\Gamma(1+ \delta)\Gamma(1+\delta+k)},
\end{align*}
where the last two equations follow from the definition of the $_2F_1$ hypergeometric function and the Gauss formula \cite{andrews1999special}. Hence for $k\ge 0$ the Fourier coefficient of $e^{i k \theta}$ of the function $h_\delta$ is given by
\begin{align*}
   \frac{1}{2\pi} \frac{\Gamma(1+\delta)^2}{\Gamma(1+2\delta)} \cdot  (-\delta)^{\uparrow k}  \frac{ \Gamma(1+2\delta)}{\Gamma(1+ \delta)\Gamma(1+\delta+k)}=\frac{1}{2\pi}\frac{ (- \delta)^{\uparrow k} }{(1+\delta)^{\uparrow k}}. 
\end{align*}
Since $h_\delta(\theta)=h_\delta(2\pi-\theta)$, the coefficient of $e^{-i k \theta}$ is given by the same expression.
\end{proof}

Extending $h_\delta(\theta)$ periodically on $\R$, the identity \eqref{eq:hdelta_Fourier} holds for all $\theta\in \R$. This shows that the right side of \eqref{eq:rho1_sum} can be simplified as 
\[
E[h_\delta(\alpha_\lambda(0))]={\frac{1}{2\pi}} \frac{\Gamma(1+\delta)^2}{\Gamma(1+2\delta)}E|1-e^{-i \alpha_\lambda(0)}|^{2\delta}={\frac{2^{2\delta}}{2\pi}} \frac{\Gamma(1+\delta)^2}{\Gamma(1+2\delta)}E|\sin(\alpha_\lambda(0)/2)|^{2\delta}
\]

\begin{lemma}\label{lem:exp_count}
Suppose that $g : \R\to \R$ is a continuous, strictly increasing deterministic function with $g(0)=0$. Fix $\delta>0$, and let  $\Theta\in[0,2\pi)$ be a random variable with density  \eqref{eq:Theta}. Consider the process
\[
\Xi=\{\lambda\in \R: g(\lambda)=\Theta \mod 2\pi\}.
\]
Then for $\lambda>0$ the expected number of points of $\Xi$ in $[0,\lambda]$ is given by 
\begin{align}\label{eq:Xi_counting}
E[\Xi[0,\lambda]]&=\frac{g(\lambda)}{2\pi}+\frac{1}{\pi}\sum_{k=1}^\infty  \frac{(-\delta)^{\uparrow k}}{(1+ \delta)^{\uparrow k}}\frac{\sin( k g(\lambda))}{k}.
\end{align}
\end{lemma}

\begin{proof}
Since $\lambda>0$, $g(0)=0$, and $g$ is continuous and strictly increasing, we have
\[
\Xi[0,\lambda]=\big|\{k\in \Z_{\ge 0}: \Theta+2k\pi\le g(\lambda) \}\big|=\lfloor\tfrac{g(\lambda)}{2\pi} \rfloor+\ind \left(g(\lambda)-2\pi \lfloor\tfrac{g(\lambda)}{2\pi}\rfloor\ge \Theta \right),
\]
and
\[
E[\Xi[0,\lambda]]=\lfloor\tfrac{g(\lambda)}{2\pi} \rfloor+P\left(g(\lambda)-2\pi \lfloor\tfrac{g(\lambda)}{2\pi}\rfloor\ge \Theta \right).
\]
Integrating the result of Lemma \ref{lem:h_delta} we get that for $0\le x<2\pi$ 
\begin{align} \label{eq:CDF_h}
    \int_0^x  h_\delta(\theta) d\theta=\frac{x}{2\pi}+\frac1{\pi} \sum_{k=1}^\infty  \frac{(-\delta)^{\uparrow k}}{(1+\delta)^{\uparrow k}}\cdot \frac{\sin(k x)}{k}. 
\end{align}
Applying this with $x=g(\lambda)-2\pi \lfloor\tfrac{g(\lambda)}{2\pi}\rfloor$, and noting that $\sin(k x)=\sin(k g(\lambda))$ if $k\in \ZZ_+$ the statement of the lemma follows.
\end{proof}

\begin{definition}\label{def:Gd}
Define the operator $\cG_\delta:\CC^{\ZZ_{+}}\to \CC^{\ZZ_{+}}$ acting on sequences $\underline{z}=(z_1,z_2,\dots)$ as 
\begin{align}\label{def:G_delta}
    \left(\cG_\delta(\underline{z})\right)_k=\tfrac12 k(k+\delta)z_{k-1}-k^2 z_k+\tfrac12 k(k-\delta) z_{k+1}, \qquad k\ge 1,
\end{align}
where we set $z_0$ to be 1. 
\end{definition}

The next claim follows directly from the definition of the Pochhammer symbol and partial summation. 
\begin{claim} \label{cl:partial_sum} For any integer $m\ge 1$ and $\underline{z}\in \CC^{\ZZ_{+}}$ we have
    \begin{align}
          \sum_{k=1}^m \frac{(-\delta)^{\uparrow k}}{k (1+\delta)^{\uparrow k}}\left(\cG_\delta (\underline{z})\right)_k=\tfrac{\delta}{2}(z_1-1)-\tfrac{1}{2}(m-\delta) \frac{(-\delta)^{\uparrow m}}{ (1+\delta)^{\uparrow m}}(z_m-z_{m+1}).
    \end{align}
\end{claim}

\begin{proof}[Proof of Theorem \ref{thm:1p_HP}]
It is sufficient to prove the statement for $\lambda>0$, the other case follows by symmetry.  Recall the construction of $\HPb$ from \eqref{eq:HP} and the properties of $\alpha_\lambda$ from Proposition \ref{prop:alphaSDE}.
Since the process $\alpha_\lambda$ is independent of $\Theta$, and $\lambda\mapsto \alpha_\lambda(0)$ satisfies the  conditions required for $g(\lambda)$ in Lemma \ref{lem:exp_count}, 
we obtain
\begin{align}\label{eq:intHP}
   E[\HPb[0,\lambda]]&=\frac1{2 \pi} E[\alpha_\lambda(0)]+\frac{1}{\pi}\sum_{k=1}^\infty \frac{(-\delta)^{\uparrow k}}{k (1+\delta)^{\uparrow k}}  \Im E[e^{i k \alpha_\lambda(0)}].
\end{align}
The sum on the right is absolutely convergent by the bound \eqref{eq:coeff_bnd}. To prove Theorem \ref{thm:1p_HP}, it is sufficient to show that the right side of \eqref{eq:intHP} can be obtained as the integral of the right side of \eqref{eq:rho1_sum} on $[0,\lambda]$. We will do this using the differential equations of Lemma \ref{lem:qk} obtained from It\^o's formula, together with a similar equation for $E [\alpha_\lambda(u)]$. The statement will follow after a cancellation resulting from the partial summation statement of  Claim \ref{cl:partial_sum}.

By Proposition \ref{prop:alphaSDE} the function $\lambda\mapsto \alpha_\lambda(0)$ is monotone and continuous in $\lambda$ with $\alpha_0(0)=0$. Since $\Theta$ is absolutely continuous and independent of $\alpha_\lambda$, \eqref{eq:HP} implies that with probability one $0\notin \HPb$. 
Hence by the monotone convergence theorem 
\begin{align*}
   E[\HPb[0,\lambda]]&=\lim_{\eps\to 0^+}  E[\HPb[0,\lambda]-\HPb[0,\eps  ]]\\
  &=\lim_{\eps\to 0^+} \left(\frac1{2\pi} E[\alpha_\lambda(0)-\alpha_{\eps}(0)]+\frac1{\pi}\sum_{k=1}^\infty \frac{(-\delta)^{\uparrow k}}{k (1+\delta)^{\uparrow k}} \Im E[e^{i k \alpha_\lambda(0)}-e^{i k \alpha_{\eps}(0)}]\right).
\end{align*}
Set $\eps=\lambda e^{\beta \nu/4}$ with $\nu<0$. Then by \eqref{eq:scale_inv} of Proposition \ref{prop:alphaSDE} we have $\alpha_{\eps}(0)\ed \alpha_\lambda(\nu)$, and hence
\begin{align*}
& \frac1{2\pi} E[\alpha_\lambda(0)-\alpha_{\eps}(0)]+\frac{1}{\pi}\sum_{k=1}^\infty \frac{(-\delta)^{\uparrow k}}{k (1+\delta)^{\uparrow k}}\Im E[e^{i k \alpha_\lambda(0)}-e^{i k \alpha_{\eps}(0)}]\\
&\qquad\qquad \qquad =\frac1{2\pi} E[\alpha_\lambda(0)-\alpha_{\lambda}(\nu)]+\frac{1}{\pi}\sum_{k=1}^\infty \frac{(-\delta)^{\uparrow k}}{k (1+\delta)^{\uparrow k}}\Im E[e^{i k \alpha_\lambda(0)}-e^{i k \alpha_{\lambda}(\nu)}]\\
&\qquad \qquad\qquad =\frac1{2\pi} E[\alpha_\lambda(0)-\alpha_{\lambda}(\nu)]+\lim_{m\to \infty}\frac{1}{\pi}\sum_{k=1}^m \frac{(-\delta)^{\uparrow k}}{k (1+\delta)^{\uparrow k}}\Im E[e^{i k \alpha_\lambda(0)}-e^{i k \alpha_{\lambda}(\nu)}].
\end{align*}
Note that $\Im E[e^{i k \alpha_\lambda(u)}]=\Im \tilde q_k(u, \lambda)$. Equation \eqref{eq:tqk_ODE} of Lemma \ref{lem:qk} can be rewritten as 
\[
\partial_u \tilde q_k(u,\lambda)= \left(\cG_\delta(\tilde q_j(u), j\ge 1)\right)_k+i k \lambda \tfrac{\beta}{4}e^{\frac{\beta}{4}u} \tilde q_k(u,\lambda),
\]
where the first term is the $k$th coordinate of the image of the operator $\cG_\delta$ acting on the sequence $z_j=\tilde q_j(u,\lambda), j\ge 1$. Hence, for $\nu<0$ we have
\[
\Im E[e^{i k \alpha_\lambda(0)}- e^{i k \alpha_\lambda(\nu)}]=\Im \int_\nu^0  \left[ \left(\cG_\delta(\tilde q_j(u,\lambda), j\ge 1)\right)_k+i k \lambda \tfrac{\beta}{4}e^{\frac{\beta}{4}u} \tilde q_k(u,\lambda) \right] du.
\]
Considering the integrated version of the equation  \eqref{eq:alphaSDE} on $[\nu,0]$ and taking expectations (the coefficient of the noise term is bounded, hence this is justified) leads to
\begin{align}\label{eq:Ealpha}
   E[\alpha_\lambda(0)-\alpha_\lambda(\nu)]=\lambda(1-e^{\beta \nu/4})-\delta \Im \int_\nu^0 \tilde q_1(u,\lambda) du. 
\end{align}
Hence, with a fixed $m\ge 1$, we get
\begin{align*}
    &\frac1{2\pi} E[\alpha_\lambda(0)-\alpha_{\lambda}(\nu)]+\frac{1}{\pi}\sum_{k=1}^m \frac{(-\delta)^{\uparrow k}}{k (1+\delta)^{\uparrow k}}\Im E[e^{i k \alpha_\lambda(0)}-e^{i k \alpha_{\lambda}(\nu)}]\\
    &\qquad =\frac{\lambda}{2\pi}(1-e^{\beta \nu/4})-\frac{\delta}{2\pi} \Im \int_{\nu}^0 \tilde q_1(u,\lambda)du+\frac{1}{\pi}\Im  \int_{\nu}^0 \sum_{k=1}^m \frac{(-\delta)^{\uparrow k}}{k (1+\delta)^{\uparrow k}} \left(\cG_\delta (\tilde q_j(u,\lambda),j\ge 1)\right)_k du\\
    &\qquad\qquad
    +\frac{\lambda}{\pi} \int_{\nu}^0 \sum_{k=1}^m \frac{\beta}{4}e^{\frac{\beta}{4}u} \frac{(-\delta)^{\uparrow k}}{ (1+\delta)^{\uparrow k}}  \Re[\tilde q_k(u,\lambda)] du.
\end{align*}
Using Claim \ref{cl:partial_sum} for the sequence $a_j=\tilde q_j(u,\lambda)$ we get for any $m\ge 1$,
\begin{align}\notag  &-\frac{\delta}{2\pi} \Im \tilde q_1(u,\lambda)+\frac{1}{\pi}\Im   \sum_{k=1}^m  \frac{(-\delta)^{\uparrow k}}{k (1+\delta)^{\uparrow k}}\left(\cG_\delta (\tilde q_j(u,\lambda), j\ge 1)\right)_k\\&\notag \qquad  = -\frac{\delta}{2\pi} \Im \tilde q_1(u,\lambda)+\frac{\delta}{2\pi}\Im[\tilde q_1(u,\lambda)-1]-\frac1{2\pi} (m-\delta)\frac{(-\delta)^{\uparrow m}}{ (1+\delta)^{\uparrow m}}\Im[\tilde q_m(u,\lambda)-\tilde q_{m+1}(u,\lambda)]\\
   &\qquad  =-\frac1{2\pi} (m-\delta)\frac{(-\delta)^{\uparrow m}}{ (1+\delta)^{\uparrow m}}\Im[\tilde q_m(u,\lambda)-\tilde q_{m+1}(u,\lambda)].\label{eq:error}
\end{align}
By the estimate  \eqref{eq:coeff_bnd}, the term \eqref{eq:error} is uniformly bounded by $c m^{-2\delta}$. Hence for any fixed $\nu<0$, this term will vanish when we integrate both sides of the equation on $[\nu,0]$ and then let $m\to \infty$.  This leads to 
\begin{align*}
&   \frac1{2\pi} E[\alpha_\lambda(0)-\alpha_{\lambda}(\nu)]+\lim_{m\to \infty}\frac{1}{\pi}\sum_{k=1}^m \frac{(-\delta)^{\uparrow k}}{k (1+\delta)^{\uparrow k}}\Im E[e^{i k \alpha_\lambda(0)}-e^{i k \alpha_{\lambda}(\nu)}]\\
&\qquad\qquad =\frac{\lambda}{2\pi}(1-e^{\beta \nu/4}) +\frac{\lambda}{\pi} \lim_{m\to \infty}\int_{\nu}^0 \sum_{k=1}^m \frac{\beta}{4}e^{\frac{\beta}{4}u} \frac{(-\delta)^{\uparrow k}}{ (1+\delta)^{\uparrow k}}  \Re[\tilde q_k(u,\lambda)] du\\
&\qquad\qquad =\frac{\lambda}{2\pi}(1-e^{\beta \nu/4}) + \frac{\lambda}{\pi} \int_\nu^0 \sum_{k=1}^\infty  \frac{\beta}{4}e^{\frac{\beta}{4}u} \frac{(-\delta)^{\uparrow k}}{ (1+\delta)^{\uparrow k}}  \Re[q_k(e^{\beta u/4} \lambda)] du\\
&\qquad\qquad =\frac{\lambda}{2\pi}(1-e^{\beta \nu/4}) + \frac{1}{\pi}\int_\eps^\lambda \sum_{k=1}^\infty  \frac{(-\delta)^{\uparrow k}}{ (1+\delta)^{\uparrow k}}  \Re[q_k(s)] ds.
\end{align*}
The scaling identity between $\tilde q_k$ and $q_k$ given in Lemma \ref{lem:qk} implies the second equation from the first. From $|q_k(\cdot)|\le 1$ and the bound \eqref{eq:coeff_bnd} it follows that  the series is absolutely convergent. In the last equation we used the change of variable $s=\lambda e^{\frac{\beta}{4}u}$ with $\eps=\lambda e^{\frac{\beta}{4} \nu}$. Letting $\eps\to 0$,  the bounded convergence theorem and $\Re[q_k(s)]=E[\cos(k \alpha_s(0))]$ yield
\[
 E[\HPb[0,\lambda]]=\frac{\lambda}{2\pi}+\int_{0}^\lambda \frac1{\pi}\sum_{k=1}^\infty    \frac{(-\delta)^{\uparrow k}}{ (1+\delta)^{\uparrow k}}  E[\cos(k \alpha_s(0))] ds,
\] proving Theorem \ref{thm:1p_HP}.
\end{proof}

\section{Analytic properties of the pair correlation function}\label{sec:analytic}

This section proves Theorems \ref{thm:ODE_representation}, \ref{thm:beta_cont}, and \ref{thm:large_lambda}. For the first and last of these theorems we prove the corresponding version of the statement for the $\HPb$ process.

\subsection{Analytic characterization of the pair correlation function}

The main  result of this section is the following proposition, which shows that the differential equation system \eqref{eq:qk_ODE} has a unique bounded solution. 

\begin{proposition}\label{prop:ODE_char}Fix $\beta>0$, $\delta> 0$.  Suppose that the functions $q_k:(0,\infty)\to \CC$, $k\ge 1$, satisfy the system of equations \eqref{eq:qk_ODE} with $q_0(\lambda)=1$ and the boundary condition \eqref{eq:qk0}. If $\sup\limits_{\lambda>0, k\ge 1} |q_k(\lambda)|<\infty$ then $q_k(\lambda)=E[\exp(i k \alpha_\lambda(0))]$ where $\alpha_\lambda(u)$ is the solution of  \eqref{eq:alphaSDE}.  
\end{proposition}
Next we show how Proposition \ref{prop:ODE_char} implies Theorem \ref{thm:ODE_representation}. 
\begin{proof}[Proof of Theorem \ref{thm:ODE_representation}]
Set $\delta=\beta/2$. Theorem \ref{thm:sineb_corr}  shows that if $\alpha_\lambda(u)$ is the solution of \eqref{eq:alphaSDE_b} and $q_k(\lambda)=E[\exp(i k \alpha_\lambda(0))]$ then \eqref{eq:twopoint_cor2} holds.  Lemma \ref{lem:qk} shows that $q_k, k\ge 1$ satisfy the system of differential equations \eqref{eq:EQk}, $k\ge 1$ with
the appropriate boundary condition at $\lambda\to 0^+$.  Proposition \ref{prop:ODE_char} shows that these functions give the unique uniformly 
bounded solution of the system with the prescribed boundary condition, which finishes the proof of the theorem.  
\end{proof}
The  same argument using Theorem \ref{thm:1p_HP} provides an analytic characterization of the one-point function $\rho_{\beta,\delta}^{(1)}$ of the $\HPb$ process for $\delta>0$ in terms of the ODE system \eqref{eq:qk_ODE}.\smallskip

The proof of Proposition \ref{prop:ODE_char} relies on a weighted energy estimate to show that uniformly bounded solutions of the homogeneous version of the system of equations \eqref{eq:qk_ODE} must be constant zero. 

\begin{proof}[Proof of Proposition \ref{prop:ODE_char}]
By Lemma \ref{lem:qk} we only have to prove that the system of equations \eqref{eq:qk_ODE} cannot have two different, uniformly bounded solutions with the prescribed boundary conditions. 
For this, it is sufficient to show that the difference of two such solutions must be constant 0. More precisely, it is enough to prove the following. Suppose that $q_k: (0,\infty)\to \CC$, $k\ge 1$ satisfy the system of equations \eqref{eq:qk_ODE} with $q_0(\lambda)=0$ and the boundary condition $\lim_{\lambda\to 0^+} q_k(\lambda)=0, k\ge 1$.  If $\sup\limits_{\lambda>0, k\ge 1} |q_k(\lambda)|<\infty$ then $q_k(\lambda)=0$ for all $\lambda>0, k\ge 1$. 

We construct a sequence $c_k, k\ge 1$ satisfying $c_k>0$ and $\sum_{k=1}^\infty \frac{c_k}{k}<\infty$ for which the function $R(t)=\sum_{k=1}^\infty \frac{c_k}{k} |q_k(t)|^2$ is  decreasing on $(0,\infty)$. From our assumptions and the bounded convergence theorem, we get  $\lim_{t\to 0^+} R(t)=0$.  This implies $R(t)=0$ for all $t > 0$, from which it follows that  $q_k(\lambda)=0$ for all $\lambda>0, k\ge 1$.

From \eqref{eq:qk_ODE} it follows that for $\lambda>0$ we have 
\begin{align}
    \frac{1}{k} \frac{d}{d\lambda} |q_k(\lambda)|^2=\frac{4}{\beta \lambda} \left((k+\delta) \Re \left[\bar{q}_{k-1}(\lambda) q_k(\lambda)\right]-2k |q_k(\lambda)|^2+(k-\delta) \Re\left[\bar{ q}_k(\lambda) q_{k+1}(\lambda)\right] \right),
\end{align}
where now $q_0(\lambda)=0$.  For a fixed $N\ge 1$ this gives
\begin{align}
\left( \sum_{k=1}^N\frac{c_k}{k}|q_k(\lambda)|^2\right)'
&=
\frac{4}{\beta \lambda}
\left(
-2\sum_{k=1}^N c_k k|q_k(\lambda)|^2
+\sum_{k=1}^{N-1}
\big(
c_k(k-\delta)+c_{k+1}(k+1+\delta)
\big)
\Re[\bar{q}_k(\lambda)q_{k+1}(\lambda)]\right)
\nonumber\\
&\qquad
+\frac{4}{\beta \lambda} c_N(N-\delta)
\Re[\bar{q}_N(\lambda)q_{N+1}(\lambda)].
\label{eq:RN-derivative}
\end{align}
Set $k_*=\lceil \delta \rceil$ and define
\begin{align}
    c_k=\begin{cases}
\Gamma \left(k+\delta+1\right)^{-1} \Gamma \left(\delta+1-k\right)^{-1}, \qquad &\text{if } 1\le k\le k_*,\\
r_\delta  \Gamma(k-\delta) \Gamma(\delta+1+k)^{-1}
        , \qquad &\text{if } k\ge 1+k_*.
    \end{cases}
\end{align}
We define the constant $r_\delta$ as 
\[
r_\delta=\begin{cases}
    \Gamma(\delta+1-k_*)^{-1} \Gamma(k_*-\delta)^{-1} \qquad &\delta\notin \Z_{+},\\
    1\qquad &\delta \in \Z_{+}.
\end{cases}
\]
With these definitions we have $c_k>0$ for all $k$, and $c_k \sim r_\delta k^{-1-2\delta}$ for large $k$. This implies that $\sum_{k=1}^\infty \frac{c_k}{k}<\infty$. 

We record the following identities that are consequences of $x \Gamma(x)=\Gamma(x+1)$:
\begin{align}
    &c_k(k-\delta)=-c_{k+1}(k+1+\delta), \qquad 1\le k\le k_*-1, \label{eq:ck_id1}
\\
    &c_k(k-\delta)=c_{k+1}(k+1+\delta) , \qquad \phantom{-}k_*+1\le k.\label{eq:ck_id2}
\end{align}
Note that when $\delta$ is not an integer, \eqref{eq:ck_id2} also holds for $k=k_*$.

Using \eqref{eq:ck_id1}, for $N>k_*$ we obtain
\begin{align}\notag
   & -2\sum_{k=1}^N c_k k|q_k(\lambda)|^2
+\sum_{k=1}^{N-1}
\big(
c_k(k-\delta)+c_{k+1}(k+1+\delta)
\big)
\Re[\bar{q}_k(\lambda)q_{k+1}(\lambda)]\\
&\qquad =-2\sum_{k=1}^N c_k k|q_k(\lambda)|^2
+\sum_{k=k_*}^{N-1}
\big(
c_k(k-\delta)+c_{k+1}(k+1+\delta)
\big)
\Re[\bar{q}_k(\lambda)q_{k+1}(\lambda)]\label{eq:c_k_cancel}
\end{align}
For $k\ge k_*$ the coefficient $c_{k} (k-\delta)+c_{k+1}(k+1+\delta)$ is positive.  From the Cauchy-Schwarz inequality and  \eqref{eq:ck_id2},  we have for $k\ge k_*+1$,
\begin{align}\left(c_{k} (k-\delta)+c_{k+1}(k+1+\delta)\right)\Re \left[\bar q_{k}(\lambda) q_{k+1}(\lambda)\right]\le c_{k} (k-\delta) |q_k(\lambda)|^2+c_{k+1}(k+1+\delta) |q_{k+1}(\lambda)|^2.
\label{eq:c_k_CauchySch}
\end{align}
When $\delta$ is not an integer (i.e.~$k_*>\delta)$, this inequality also holds for $k=k_*$, because of the comment after \eqref{eq:ck_id2}. In this case \eqref{eq:RN-derivative}, \eqref{eq:c_k_cancel}, and \eqref{eq:c_k_CauchySch} yield
\begin{align}
\left( \sum_{k=1}^N\frac{c_k}{k}|q_k(\lambda)|^2\right)'
\le
\frac{4}{\beta \lambda}
&\left(-2\sum_{k=1}^{k_*-1}c_k k|q_k|^2(\lambda)
-c_{k_*}(k_*+\delta)|q_{k_*}|^2 (\lambda)\right)
\nonumber\\
&-\frac{4}{\beta \lambda}c_N(N-\delta)|q_N(\lambda)|^2
+\frac{4}{\beta \lambda} c_N(N-\delta)\Re[\overline q_N(\lambda)q_{N+1}(\lambda)].
\label{eq:RN-pointwise-noninteger}
\end{align}
The functions $q_k, k\ge 1$ are uniformly bounded  by some constant $M$ our assumption. From $c_k\sim r_\delta k^{-1-2\delta}$ we can now find a  constant $c_0$ (depending only on $\beta, \delta$, and $M$) so that 
\[
\left|-\frac{4}{\beta \lambda}c_N(N-\delta)|q_N(\lambda)|^2
+\frac{4}{\beta \lambda} c_N(N-\delta)\Re[\overline q_N(\lambda)q_{N+1}(\lambda)]\right|
\le c_0 \lambda^{-1} N^{-2\delta}, 
\]
uniformly in  $\lambda>0$ and $N>k_*$. This yields the following uniform bound for $0<s<t$:
\begin{align}
\sum_{k=1}^N\frac{c_k}{k}|q_k(t)|^2-\sum_{k=1}^N\frac{c_k}{k}|q_k(s)|^2
&\le c_0 N^{-2\delta}\log(t/s).
\label{eq:RN-integrated}
\end{align}
Now let $N\to\infty$.  Since $\sum_{k=1}^\infty \frac{c_k}{k}<\infty$ and $q_k$ is uniformly bounded, we obtain that $R(t)-R(s)\le 0$ for $0<s<t$, as required. 

The proof in the case when $\delta$ is an integer (i.e.~$k_*=\delta$) is similar. In this case we have
\[
c_{k_*} (k_*-\delta)+c_{k_*+1}(k_*+1+\delta)=c_{k_*+1}(2k_*+1).
\]
Hence \eqref{eq:RN-derivative}, \eqref{eq:c_k_cancel}, and \eqref{eq:c_k_CauchySch} yield
\begin{align*}
  \left( \sum_{k=1}^N\frac{c_k}{k}|q_k(\lambda)|^2\right)'
&\leq 
\frac{4}{\beta \lambda}\left(-2\sum_{k=1}^{k_*-1} c_k k|q_k(\lambda)|^2 -  (2k_* c_{k_*}-c_{k_*+1} (k_*+1/2))|q_{k_*}(\lambda)|^2 \right)\\
&\quad -\frac{4}{\beta \lambda}c_N(N-\delta)|q_N(\lambda)|^2
+\frac{4}{\beta \lambda} c_N(N-\delta)\Re[\overline q_N(\lambda)q_{N+1}(\lambda)].
\end{align*}
From the definitions, we have
\[
2k_* c_{k_*}-c_{k_*+1} (k_*+1/2)=\frac{4 k_*-1}{2 \Gamma (2 k_*+1)}>0.
\]
Hence we can follow the steps of the proof of the previous case to get \eqref{eq:RN-integrated} and $R(t)\le R(s)$ for $0<s<t$, as required.  
\end{proof}

\subsection{Regularity of the pair correlation function}

We turn to the proof of Theorem \ref{thm:beta_cont}. We will use the following proposition, which will be proved in Section \ref{subs:continuity} of the Appendix.

\begin{proposition}\label{prop:sde_cont}
Let $Z$ be a two-sided Brownian motion.
There is a process $\alpha_{\lambda,\beta}(u)$, $\lambda\in \R$, $\beta>0$, $u\in \R$ that is adapted to the filtration of $Z$ as a process in $u$, analytic in $\lambda$ for fixed $u$ and  $\beta$, continuous in  $\beta$ for fixed $u$ and $\lambda$, and has the same distribution as the unique strong solution $u\mapsto \alpha_{\lambda}(u)$ of  \eqref{eq:alphaSDE_b} for a fixed $\beta$.     
\end{proposition}

\begin{proof}[Proof of Theorem \ref{thm:beta_cont}]
Consider the process $\alpha_{\lambda,\beta}(u)$ given by  Proposition \ref{prop:sde_cont}, then from \eqref{def:qk} we have 
\[
q_{k,\beta,\beta/2}(\lambda)=E [\exp(i k \alpha_{\lambda,\beta}(0))].
\]
From Proposition \ref{prop:sde_cont} it follows that these functions are continuous  in $\lambda$ and continuous in $\beta$, and by definition we have $|q_k(\beta,\lambda)|\le 1$.  By Theorem \ref{thm:sineb_corr} we have
\begin{align}\label{eq:qqq}
    \rho_\beta^{(2)}(0,\lambda)=\frac{1}{4\pi^2}+\frac{1}{2\pi^2}\sum_{k=1}^\infty  \frac{(-\beta/2)^{\uparrow k}}{(1+\beta/2)^{\uparrow k}} \Re q_{k,\beta, \beta/2}(\lambda).
\end{align}
The first part of the theorem now follows:  the partial sums in \eqref{eq:qqq} are continuous in $\lambda$ and in $\beta$, and these continuous functions converge locally uniformly to  $\rho_\beta^{(2)}(0,\lambda)$ because of the summable bound \eqref{eq:coeff_bnd}.

By Lemma \ref{lem:qk}, the functions $q_{k,\beta,\beta/2}(\lambda)$ satisfy the ODEs \eqref{eq:qk_ODE}  with $\delta=\frac{\beta}{2}$. Integrating the ODE on $[\lambda_1,\lambda_2]$ and using  $|q_{k,\beta,\beta/2}(\lambda)|\le 1$, it follows that for $k\ge 1$ and $0<\lambda_1<\lambda_2$, we have
\begin{align}\label{eq:Reqk_bnd}
|\Re q_{k,\beta,\beta/2}(\lambda_1)-\Re q_{k,\beta,\beta/2}(\lambda_2)|\le c_\beta(k^2 \lambda_1^{-1}+k) |\lambda_2-\lambda_1|, 
\end{align}
where $c_\beta$  is a locally bounded function of $\beta$ that does not depend on $\lambda$. When $\beta>2$ is fixed,  the bound \eqref{eq:coeff_bnd} together with 
Theorem \ref{thm:sineb_corr} and \eqref{eq:Reqk_bnd} implies
\begin{align}
\label{bd:rho_lipschitz}
    |\rho_\beta^{(2)}(0,\lambda_2)-\rho_\beta^{(2)}(0,\lambda_1)|\le c_\beta'  (1+\lambda_1^{-1})|\lambda_2-\lambda_1|, \qquad 0<\lambda_1<\lambda_2.
\end{align}
Again, $c_\beta'$ is a locally bounded function of $\beta$ that does not depend on $\lambda$. Suppose that $(\beta_n, \lambda_n)\to (\beta,\lambda)$ with $\beta_n, \beta>2$ and $\lambda_n, \lambda>0$. We have
\[
|\rho_{\beta_n}^{(2)}(0,\lambda_n)-\rho_\beta^{(2)}(0,\lambda)|\le |\rho_\beta^{(2)}(0,\lambda)-\rho_{\beta_n}^{(2)}(0,\lambda)|+|\rho_{\beta_n}^{(2)}(0,\lambda)-\rho_{\beta_n}^{(2)}(0,\lambda_n)|. 
\]
Since $\rho_\beta^{(2)}(0,\lambda)$ is continuous in $\beta$, the first term converges to 0. By \eqref{bd:rho_lipschitz}, the second term converges to 0 as well. 
This completes the proof of the joint continuity in the $\beta>2, \lambda>0$ region.
\end{proof}

\subsection{Large $\lambda$ bounds}\label{sec:large_lambda}

We prove the following theorem for the asymptotic behavior of the one-point function of the $\HPb$ process. 

\begin{theorem}[Asymptotics of the one-point  function of $\HPb$]\label{thm:HPlarge_lambda} 
Fix $\beta>0$, $\delta>0$. There is a constant $c$ depending only on $\beta$ and $\delta$ so that  the following bound holds for $\lambda>1$: 
\begin{align}\label{eq:HPlarge_lambda}
     \left|\rho^{(1)}_{\beta,\delta}(\lambda)-\frac{1}{2\pi}\right|\le c\left(\lambda^{-2}+\lambda^{-\delta}+\lambda^{-4/\beta}+\lambda^{-2}\log(\lambda)(\ind(\delta=2)+\ind(\beta=2))\right).
\end{align} 
\end{theorem}
Theorem \ref{thm:HPlarge_lambda} immediately implies Theorem \ref{thm:large_lambda}.
\begin{proof}[Proof of Theorem \ref{thm:large_lambda}]
From Theorem \ref{thm:sineb_corr} and the estimate \eqref{eq:coeff_bnd}, we can obtain a $\beta$-dependent uniform upper bound on $\left|\rho^{(2)}_{\beta}(0,\lambda)-\frac{1}{4 \pi^2}\right|$ for all $\lambda$. Hence it is sufficient to prove \eqref{eq:large_lambda} for $\lambda> 1$. 
Proposition \ref{prop:sine_HP} and    Theorem \ref{thm:HPlarge_lambda} with $\delta=\beta/2$ yield
\begin{align*}
     \left|\rho^{(2)}_{\beta}(0,\lambda)-\frac{1}{4 \pi^2}\right|\le c\left(\lambda^{-2}+\lambda^{-\beta/2}+\lambda^{-4/\beta}+\lambda^{-2}\log(\lambda)\ind(\beta\in\{2,4\})\right), \qquad \lambda> 1.
\end{align*} 
Since $\min(\beta/2, 4/\beta)<2$, the terms on the right are bounded by a constant multiple of  $\lambda^{-\beta/2}+\lambda^{-4/\beta}$, proving the theorem.
\end{proof}

The main ingredient for proving Theorem \ref{thm:HPlarge_lambda} is the following lemma. 

\begin{lemma}\label{lem:qk_bounds}
Fix $\beta>0$, $ \delta>0$, and let $q_k$ be as defined in  \eqref{def:qk}. There is a constant $c$ depending only on $\beta$ and $\delta$ so that the following bounds hold for $\lambda> 1$:
\begin{align}
      |\Re q_1(\lambda)|& \le c(\lambda^{-2}+\lambda^{-4/\beta}+\lambda^{-2} \log(\lambda)\ind(\beta=2)),      \label{eq:eialpha_1} \\
      | q_2(\lambda) |&\le c( \lambda^{-2}+\lambda^{-4/\beta}),      \label{eq:eialpha_2}  \\
      |q_k(\lambda)|& \le c(  k^2 \lambda^{-4k^2/\beta} +k^4 \lambda^{-2}+\lambda^{-4/\beta}),\qquad \qquad  k\geq 3.
      \label{eq:eialpha_k} 
\end{align}
\end{lemma}
We first show how Theorem \ref{thm:HPlarge_lambda} follows from the lemma.
\begin{proof}[Proof of Theorem \ref{thm:HPlarge_lambda}]
In the arguments below, $c$ is always a positive constant depending only on $\beta, \delta$, but its value might change from line to line. 

From Theorem \ref{thm:1p_HP} and the bound \eqref{eq:coeff_bnd} we obtain
\begin{align}\label{eq:trunc_bnd}
   \left|\rho^{(1)}_{\beta,\delta}(\lambda)-\frac{1}{2\pi}\right|\le c\sum_{k=1}^\infty  k^{-1-2\delta} \left|\Re q_k(\lambda)\right|.
\end{align}
We bound $\left|q_k(\lambda)\right|$ by 1 for $k>\lambda^{1/2}$, and use the bounds \eqref{eq:eialpha_1}-\eqref{eq:eialpha_k} otherwise.
\begin{align*}
    \left| \rho^{(1)}_{\beta,\delta}(\lambda)-\frac{1}{2\pi}\right|
    &\leq c \big(\lambda^{-2}+\lambda^{-4/\beta}+\lambda^{-2} \log(\lambda) \ind(\beta=2)\big)
      \\
      &\qquad + c \sum_{3\le k\le \sqrt{\lambda}} k^{-1-2\delta}(k^2 \lambda^{-4k^2/\beta} +k^4 \lambda^{-2}+\lambda^{-4/\beta}) + \sum_{k>\sqrt{\lambda}} k^{-1-2\delta}\\
      &\le c(\lambda^{-2}+\lambda^{-4/\beta}  +  \lambda^{-\delta}+\lambda^{-2}\log(\lambda)(\ind(\delta=2)+\ind(\beta=2))).
\end{align*}
This concludes the proof.
\end{proof}

\begin{proof}[Proof of Lemma \ref{lem:qk_bounds}]
The proof relies on the fact that the functions $q_k(\lambda), k\ge 1$ are bounded by one, and satisfy the differential equations \eqref{eq:qk_ODE} with $q_0(\lambda)=1$. 

Introduce the (local) notation 
\begin{align}
 \breve {q}_k(\lambda)=e^{-ik\lambda} \q_k(\lambda), \qquad k\in \Z_{\ge 0}, \lambda>0. 
\end{align}
Note that $\breve  q_0(\lambda)=q_0(\lambda)=1$. By \eqref{eq:qk_ODE} these functions satisfy the differential equations
\begin{align}
\lambda \breve  q'_{k}(\lambda)&=a_{k,-} e^{-i \lambda} \breve  q_{k-1}(\lambda)-a_k \breve  q_{k}(\lambda)+a_{k,+}e^{i \lambda} \breve  q_{k+1}(\lambda), \qquad k\ge 1,
\label{eq:tq_ODE}
\end{align}
where we have
\begin{align}
    a_{k,-} &= \tfrac{2}{\beta} k(k+\delta), \qquad  a_k = \tfrac{4}{\beta}k^2, \qquad a_{k,+} = \tfrac{2}{\beta}k(k-\delta).
\end{align}
Throughout this proof we will use the notation $\lesssim$ to suppress constant multipliers that only depend on $\beta$ and $\delta$, and  may vary from line to line. For example, 
\begin{align}\label{eq:coeff_simple}
    \max(|a_{k,-}|, |a_k|, |a_{k,+}|)\lesssim k^2. 
\end{align}
For $\lambda>1$, equation \eqref{eq:tq_ODE} can be rewritten using variation of parameters as 
\begin{align}
\breve q_k(\lambda)
={}&
\lambda^{-a_k}\left(\breve q_k(1)
+a_{k,-}
\int_1^\lambda
e^{-iu}u^{a_k-1}\breve q_{k-1}(u)\,du
+a_{k,+}
\int_1^\lambda
e^{iu}u^{a_k-1}\breve q_{k+1}(u)\,du\right).
\label{eq:tq_integral}
\end{align}
We can further rewrite the two integrals in \eqref{eq:tq_integral} using integration by parts to obtain
\begin{align}
\breve q_k(\lambda)
={}&
\lambda^{-a_k}\breve q_k(1)
+a_{k,-}\lambda^{-a_k}
i e^{-iu}u^{a_k-1}\breve q_{k-1}(u)
\Bigg|_{1}^{\lambda}
-a_{k,+}\lambda^{-a_k}
i e^{iu}u^{a_k-1}\breve q_{k+1}(u)
\Bigg|_{1}^{\lambda}
\nonumber\\
&-a_{k,-}\lambda^{-a_k}
\int_1^\lambda
i e^{-iu}u^{a_k-2}
\left(
(a_k-1)\breve q_{k-1}(u)
+u\breve q'_{k-1}(u)
\right)\,du
\nonumber\\
&+a_{k,+}\lambda^{-a_k}
\int_1^\lambda
i e^{iu}u^{a_k-2}
\left(
(a_k-1)\breve q_{k+1}(u)
+u\breve q'_{k+1}(u)
\right)\,du.
\nonumber
\end{align}
Using \eqref{eq:tq_ODE} and the fact that $\breve q_0(\lambda)=1$, we can replace $\breve q'_{k-1}(u), \breve q'_{k+1}(u)$ in the integrals above. Taking absolute values in the resulting expressions, and using the bound $|\breve q_k(\lambda)|=|q_k(\lambda)|\le 1$ and \eqref{eq:coeff_simple}, we  obtain the following upper bound for $k\ge 3$, $\lambda>1$: 
\begin{align}
\label{bd:tq_k_integral_bd}
 & \left|  \tq_k(\lambda)\right| \lesssim k^2 \lambda^{-a_k}+k^2 \lambda^{-1} (|\tq_{k-1}(\lambda)|+|\tq_{k+1}(\lambda)|) \nonumber \\
&\qquad\qquad+ k^4 \lambda^{-a_k} \int_1^\lambda u^{a_k-2} \left(|\tq_{k-2}(u)|+\cdots+|\tq_{k+2}(u)|\right)du,
\end{align}
and the following additional bounds for  $\lambda>1$: 
\begin{align}
\label{bd:tq_2_integral_bd}
 & \left| \tq_2(\lambda)\right|
 \lesssim
 \lambda^{-a_2}
 +\lambda^{-1}\bigl(|\tq_{1}(\lambda)|+|\tq_{3}(\lambda)|\bigr)\nonumber\\
&\qquad\qquad
+\lambda^{-a_2}
\left|
\int_1^\lambda e^{-2iu}u^{a_2-2}\,du
\right|
+\lambda^{-a_2}
\int_1^\lambda
u^{a_2-2}
\left(
|\tq_{1}(u)|+\cdots+|\tq_{4}(u)|
\right)\,du,\\
\label{bd:tq_1_integral_bd}
&\left|\tq_1(\lambda)
-i a_{1,-}e^{-i\lambda}\lambda^{-1}
+i a_{1,-}(a_1-1)\lambda^{-a_1}
\int_1^\lambda
e^{-iu}u^{a_1-2}\,du
\right|
\nonumber\\
&\qquad\qquad \lesssim
\lambda^{-a_1}+ |a_{1,+}| \lambda^{-1}|\tq_2(\lambda)|+
|a_{1,+}| \lambda^{-a_1}
\int_1^\lambda
u^{a_1-2}
\left(
|\tq_{1}(u)|+\cdots+|\tq_{3}(u)|
\right)\,du, \\
&|\Re q_1(\lambda)| \lesssim \lambda^{-a_1}+ |a_{1,+}| \lambda^{-1}|\tq_2(\lambda)|+  \left|a_{1,-}(a_1-1)\lambda^{-a_1}
\int_1^\lambda
e^{-iu}u^{a_1-2}\,du
\right| \nonumber \\
&\qquad\qquad +
|a_{1,+}|\lambda^{-a_1}
\int_1^\lambda
u^{a_1-2}
\left(
|\tq_{1}(u)|+\cdots+|\tq_{3}(u)|
\right)\,du .
\label{bd:req1_integral_bd}
\end{align}
For the last estimate we used $q_1(\lambda)=e^{i \lambda} \tq_1(\lambda)$ and the bound \eqref{bd:tq_1_integral_bd}.
Using $|\tq_k(\lambda)|\le 1$ in \eqref{bd:tq_k_integral_bd}-\eqref{bd:tq_1_integral_bd} leads to the following upper bound (valid for all $k\ge 1$, $\lambda>1$): 
\begin{align}\notag
    |\breve q_k(\lambda)|&\lesssim k^2 \lambda^{-a_k}+k^2 \lambda^{-1}+k^4 \lambda^{-a_k} \int_1^\lambda u^{a_k-2} du\\
    &\lesssim k^2 \lambda^{-a_k}+ k^2 \lambda^{-1} +\lambda^{-1} \log(\lambda)\ind(a_k=1).\label{eq:tlq_first_simple_bd}
\end{align}
The second inequality follows from the first by inspecting the cases $a_k<1$, $a_k=1$, $a_k>1$ separately, and using that $a_k\le 1$ implies $k\lesssim 1$. Note that in the $a_{k}>1$ case the integral can be bounded by $\frac{1}{a_{k}-1} \lambda^{a_k-1}$. For a fixed $\beta>0$ the ratio $\frac{k^2}{a_k-1}$ is uniformly bounded for all $k$ with $a_k>1$, this explains the $k^2 \lambda^{-1}$ term in \eqref{eq:tlq_first_simple_bd}.

We can obtain the required sharper bounds of the lemma by substituting the bound \eqref{eq:tlq_first_simple_bd} into \eqref{bd:tq_k_integral_bd}-\eqref{bd:req1_integral_bd}, and estimating the appearing oscillating integrals $\int_1^\lambda e^{-i j u} u^q du$ more carefully. These steps are more or less straightforward, although a bit tedious. 

To prove \eqref{eq:eialpha_k} (when $k\ge 3$), we use \eqref{eq:tlq_first_simple_bd} in \eqref{bd:tq_k_integral_bd}, and treat the terms separately. 
We  have
\begin{align*}
    &k^2 \lambda^{-a_k}+k^2 \lambda^{-1} (|\tq_{k-1}(\lambda)|+|\tq_{k+1}(\lambda)|)\\&\lesssim k^2(\lambda^{-a_k}+k^2 \lambda^{-a_{k-1}-1}+k^2 \lambda^{-2})+\lambda^{-2} \log(\lambda) \ind(a_{k-1}=1 \text{ or } a_{k+1}=1)\\
    &\lesssim k^2 \lambda^{-a_k}+k^4 \lambda^{-2}+\lambda^{-a_1}.
\end{align*}
The second inequality follows by noting that if $a_{k-1}>1$ then $\lambda^{-a_{k-1}-1}\le \lambda^{-2}$, and if $a_{k-1}\le 1$ then $k\lesssim 1$ and $\lambda^{-a_{k-1}-1}\le \lambda^{-a_1}$. Moreover, if $a_{k-1}=1$ or $a_{k+1}=1$ then $\lambda^{-2} \log(\lambda)\le \lambda^{-a_1}$.
For $j\in \{k-2,\dots, k+2\}$ we obtain the following bounds (by the monotonicity of $a_j$):
\begin{align}
 k^4 \lambda^{-a_k}   \int_1^\lambda u^{a_k-2} |\tq_j(u)|du&\lesssim k^4 \lambda^{-a_k}\int_1^\lambda (k^2 u^{a_k-a_{k-2}-2}+k^2 u^{a_k-3}+u^{a_k-3} \log(u) \ind(a_j=1)) du.\label{eq:another_bnd_00}
\end{align}
If $a_{k-2}\le 1$ then $k\lesssim 1$ and the 
 right side of \eqref{eq:another_bnd_00} can be bounded by a constant multiple of $\lambda^{-a_1}$ (with the constant only depending on $\beta$). If $a_{k-2}> 1$ then the right side of \eqref{eq:another_bnd_00} can be bounded by a constant multiple of 
 \[
 k^4 \lambda^{-a_k}\int_1^\lambda (k^2 u^{a_k-3}+u^{a_k-3} \log(u) \ind(a_j=1)) du\lesssim k^{4} \lambda^{-2}+\lambda^{-a_1},
 \]
where we can obtain the upper bound by considering the cases $a_k<2, a_k=2, a_k>2$ separately, similarly as we did in the proof of \eqref{eq:tlq_first_simple_bd}. Collecting our estimates, and noting that $|\breve q_k(\lambda)|=|q_k(\lambda)|$ gives the proof of \eqref{eq:eialpha_k}. 

To prove \eqref{eq:eialpha_2}, we start with \eqref{bd:tq_2_integral_bd}, and substitute the  bound \eqref{eq:tlq_first_simple_bd} in the right side. Integration by parts shows that 
\begin{align}\label{eq:simple_osc_int}
    \left|
\int_1^\lambda e^{-2iu}u^{a_2-2}\,du
\right|\lesssim 1+\lambda^{a_2-2}. 
\end{align}
Using the same arguments as in the proof of \eqref{eq:eialpha_k} we now obtain \eqref{eq:eialpha_2}.

Finally, the proof of \eqref{eq:eialpha_1} follows similarly, starting with \eqref{bd:req1_integral_bd} and using the bound \eqref{eq:tlq_first_simple_bd} on the right side. Bounding $\left|\int_1^\lambda
e^{-iu}u^{a_1-2}\,du
\right|$ the same way as \eqref{eq:simple_osc_int}, and using the arguments in the proof of \eqref{eq:eialpha_k} we obtain the required statement. The $\lambda^{-2} \log(\lambda)\ind(\beta=2)$ term is the result of the $\lambda^{-a_1}\int_1^\lambda u^{a_1-2} |\breve q_1(u)| du$ term when $a_1=2$. This could be removed with some additional steps, but it is not needed for our estimates.
\end{proof}

\section{Results   for integer $\delta$}\label{sec:delta_n}

This section collects the proofs of our statements for the $\Sineb$ process with $\beta$ an even integer: Theorem \ref{thm:sineb_series} and Corollary \ref{cor:small_lambda}.  Some of these statements are proved more generally in the setting of $\HPb$ when $\delta$ is a positive integer. We also demonstrate how our results yield an independent derivation of the known pair correlation in the $\beta=4$ case, and explain how to construct an ODE of order $2n$ with $\rho_{2n}^{(2)}(0,\lambda)$ as one of its solutions.

\subsection{The one-point function of $\HPb$ for integer $\delta$}

Our goal in this section is to prove Theorem \ref{thm:sineb_series} and  Corollary \ref{cor:small_lambda}, together with  similar results for $\rho_{\beta,n}^{(1)}(\lambda)$ with $n\in \Z_+$.\\
Fix $\beta>0$ and $\delta=n\in \Z_+$. In this case, the infinite sum in Theorem \ref{thm:1p_HP} becomes a finite sum with $n$ terms, and the coefficients can be simplified as well. We have 
\begin{align}\label{eq:rho_HP_n}
      \rho^{(1)}_{\beta,n}(\lambda)=\frac{1}{2\pi}+\frac{1}{\pi}\sum_{k=1}^n (-1)^k  \frac{\binom{2n}{n+k}}{\binom{2n}{n}} \Re E[e^{i k \alpha_\lambda(0)}], 
\end{align}
with $\alpha_\lambda(u)$ the strong solution of \eqref{eq:alphaSDE} with initial condition \eqref{eq:SDE_initial} and $\delta=n$. Recall the definition of $q_k(\lambda)=E[e^{i k\alpha_\lambda(0)}]$ from \eqref{def:qk}, and let
 \begin{align}
    \mathbf{q}(\lambda)=[q_1(\lambda), \dots, q_n(\lambda)]^\dag.
\end{align}   
(We suppress the dependence on $\beta$ and $n$ in the notation.)
Recall the definition of $\mathbf{A}_n$, $\mathbf{B}_n$, $\mathbf{e}_n$, $\mathbf{v}_n$  from \eqref{eq:AB}, \eqref{eq:vectors}.
Define $\mathbf{s}_k\in \CC^n , k\ge 0$
as
\begin{align}\label{eq:q_coeff_1}
    \mathbf{s}_0&=-\frac{n+1}{2}\mathbf{A}_n^{-1}\mathbf{e}_n, \qquad \mathbf{s}_k=i \left(k \mathbf{I}-\tfrac{4}{\beta} \mathbf{A}_n\right)^{-1} \mathbf{B}_n \mathbf{s}_{k-1}, \qquad k\ge 1.
\end{align}
Note that this agrees with the previous definition \eqref{eq:q_coeff} if $\beta=2n$, but we will need this sequence now for general $\beta>0$.  Lemma \ref{lem:An_eval} below shows that $A_n$ has $n$ distinct negative eigenvalues, which shows that $\mathbf{s}_k, k\ge 0$ are well defined. \\
We also introduce the notation
\begin{align}\label{eq:f}
    \mathbf{f}_n=[1,1,\dots,1]^\dag\in \R^n.
\end{align}
By  evaluating $\mathbf{A}_n \mathbf{f}_n$ we can check that $\mathbf{f}_n=\mathbf{s}_0$.

The main results of this section are the following theorem and its corollary. 

\begin{theorem}\label{thm:even_beta}
 The vector-valued function $\mathbf{q}(\lambda)$ satisfies the ODE
 \begin{align}\label{eq:ode_q_vec}
     \frac{\beta}{4} \lambda \mathbf{q}'(\lambda) = (i \frac{\beta}{4}\lambda  \mathbf{B}_n   +\mathbf{A}_n) \mathbf{q}(\lambda)+\frac{n+1}{2} \mathbf{e}_n, \qquad \mathbf{q}(0)=\mathbf{f}_n.
\end{align}
The functions $\lambda\mapsto q_k(\lambda)$ can be extended to $\CC$ as entire functions  so that \eqref{eq:ode_q_vec} holds for all $\lambda\in \CC$. With this extension, we have for all $\lambda\in \CC$
\begin{align}\label{eq:q_expansion}
    \mathbf{q}(\lambda)=\sum_{j=0}^\infty \mathbf{s}_j \lambda^j.
\end{align} 
\end{theorem}

\begin{corollary}
\label{cor:int_case_rho}
   With the notations introduced above, the one-point function  of the $\operatorname{HP}_{\beta, n}$ process is given by
\begin{align}\label{eq:int_case_rho}
    \rho^{(1)}_{\beta,n}(\lambda) =\frac{1}{2\pi}\left(1+2\;\mathbf{v}_n \cdot \Re \mathbf{q}(\lambda)\right)=\frac{1}{\pi} \sum_{j=1}^\infty \mathbf{v}_n\cdot \mathbf{s}_{2j} \lambda^{2j}.
\end{align}
\end{corollary}
Using Proposition \ref{prop:sine_HP} with $\beta=2n$ gives the analogous statements for $\rho_{2n}^{(2)}(0,\lambda)$. Theorem \ref{thm:sineb_series}  follows from Corollary \ref{cor:int_case_rho}, and Theorem \ref{thm:even_beta} leads to a description of $\rho_{2n}^{(2)}(0,\lambda)$ in terms of the ODE \eqref{eq:ode_q_vec} (with $\beta=2n$). \\
We also identify the behavior of $\rho_{\beta,n}^{(1)}(\lambda)$ near 0. 
\begin{proposition}
\label{prop:small_lambda_asym}
    For $\lambda$ near 0, we have 
       \begin{align}\label{eq:small_lambda_HP}
        \rho_{\beta,n}^{(1)}(\lambda) =\frac{1}{2\pi} 
      \binom{2n}{n}^{-1}\frac{\left(\tfrac{\beta}{2}\right)^{2n}}{\left(1+\frac{\beta}{2}\right)^{\uparrow 2n}}   \lambda^{2n}  
        + \mathcal{O}(\lambda^{2n+2}).
    \end{align} 
\end{proposition}

When $\beta=2n$, we can readily check that 
\[
 \binom{2n}{n}^{-1}\frac{\left(\tfrac{\beta}{2}\right)^{2n}}{\left(1+\frac{\beta}{2}\right)^{\uparrow 2n}} =\frac{n^{2n} (n!)^3 }{(2n)!(3n)! },
\]
which  implies Corollary \ref{cor:small_lambda}.

\begin{proof}[Proof of Theorem \ref{thm:even_beta}] 
The fact that $\mathbf{q}(\lambda)$ satisfies the ODE \eqref{eq:ode_q_vec} for $\lambda\neq 0$ follows from Lemma \ref{lem:qk} noting that when $\delta=n$, the first $n$ equations of \eqref{eq:qk_ODE} form a closed system. 
If we {assume} that the solution has an expansion of the form \eqref{eq:q_expansion} then we obtain the following equations for the coefficients of $\lambda^k, k\ge 0$:
\begin{align*}
    \mathbf{A}_n \mathbf{s_0}+\tfrac{n+1}{2}\mathbf{e}_n=\mathbf{0}, \qquad 
 k \mathbf{s}_k=i  \mathbf{B}_n\mathbf{s}_{k-1}+\tfrac{4}{\beta}\mathbf{A}_n \mathbf{s}_k , \qquad k\ge 1.
\end{align*}
The solution of this recursion is exactly the sequence $\mathbf{s}_k, k\ge 0$ given by \eqref{eq:q_coeff_1}. To show that $\mathbf{q}(\lambda)$ indeed extends to an entire function in $\lambda$ and has the expansion \eqref{eq:q_expansion}, we use an approximation procedure.

For a fixed $\nu<0$ consider the unique strong solution $\alpha_{\lambda,\nu}(u)$ of \eqref{eq:alphaSDE} on $[\nu,\infty)$ with initial condition $\alpha_{\lambda,\nu}(\nu)=0$ and $\delta=n$. Set 
\begin{align}
    \tilde q_{k,\nu}(u,\lambda)=E[e^{i k \alpha_{\lambda,\nu}(u)}].
\end{align}
Fix $\lambda\in \R$. By Proposition \ref{prop:alphaSDE} we have  $\lim\limits_{\nu\to -\infty} \alpha_{\lambda,\nu}(u)=\alpha_\lambda(u)$ almost surely, hence by the bounded convergence theorem we have 
\begin{align}
    \lim_{\nu\to -\infty} \tilde  q_{k,\nu}(0,\lambda)=q_k(\lambda). 
\end{align}
The process $e^{i k \alpha_{\lambda,\nu}(u)}$ satisfies the same SDE as $e^{i k \alpha_\lambda(u)}$, see \eqref{eq:QkSDE}. Repeating the arguments of the proof of Lemma \ref{lem:qk} we get that $\tilde q_{k,\nu}(u,\lambda)$ satisfies  the following ODE system for $u\in [\nu,\infty)$, $\lambda\in \R$:
\begin{align}\label{eq:ODEq1}
    \partial_u \tilde q_{1,\nu}(u,\lambda)&=\frac{n+1}{2}+(i \lambda \tfrac{\beta}{4}e^{\frac{\beta}{4}u}-1) \tilde q_{1,\nu}(u,\lambda)+\frac{1}{2}(1-n) \tilde q_{2,\nu}(u,\lambda),\\\notag
    \partial_u \tilde q_{k,\nu}(u,\lambda)&=\frac{k}{2} (n +k) \tilde q_{k-1,\nu}(u,\lambda)+(-k^2 + i k \lambda \tfrac{\beta}{4}e^{\frac{\beta}{4}u})\tilde q_{k,\nu}(u,\lambda)\label{eq:ODEq2}\\&\qquad\qquad  +\frac{1}{2} k (k-n )\tilde q_{k+1,\nu}(u,\lambda), \qquad \qquad 2\le k\le n,\\[3pt]
   \tilde q_{k,\nu}(\nu,\lambda)&=1, \qquad 1\le k\le n. 
\end{align}
Since the coefficient 
of $\tilde q_{n+1,\nu}$ in \eqref{eq:ODEq2} is 0,   we again obtained a closed ODE system.
Using the notation
\[
\tilde{\mathbf{q}}_{\nu}(u,\lambda)=[\tilde q_{1,\nu}(u,\lambda), \dots, \tilde q_{n,\nu}(u,\lambda)]^\dag,
\]
we can write it as 
\begin{align}\label{eq:q_nu_ODE}
     \partial_u \tilde{\mathbf{q}}_\nu(u,\lambda) &= (i \tfrac{\beta}{4}e^{\frac{\beta}{4}u} \mathbf{B}_n \lambda  +\mathbf{A}_n) \tilde{\mathbf{q}}_\nu(u,\lambda) +\frac{n+1}{2}\mathbf{e}_n, \qquad
     \tilde{\mathbf{q}}_{\nu}(\nu,\lambda) = \mathbf{f}_n.
\end{align}
For a fixed $\lambda\in \CC$, the coefficients of this ODE system are bounded for $u\in [\nu,0]$. Hence there is a unique solution on $[\nu,0]$ for each $\lambda\in \CC$, and this solution is analytic in $\lambda$ for a fixed $u$. We claim that this unique solution is given by 
\begin{align}
    \tilde{\mathbf{q}}_\nu(u,\lambda)=\sum_{j=0}^\infty \mathbf{s}_{j,\nu}(u)\lambda^j,\label{eq:q_nu_series}
\end{align}
where the vector valued functions $\mathbf{s}_{j,\nu}(u),j\ge 0$ are  given by the recursion  
\begin{align}\label{eq:s_nu_def1}
    \mathbf{s}_{0,\nu}(u) &=\frac{n+1}{2}\int_\nu^u e^{(u-s)\mathbf{A}_n} \mathbf{e}_nds+e^{(u-\nu) \mathbf{A}_n }\cdot \mathbf{f}_n,\\
    \mathbf{s}_{k,\nu}(u) &= i \frac{\beta}{4}e^{\frac{\beta}{4} u}\int_\nu^ue^{\beta/4 (s-u)}e^{(u-s)\mathbf{A}_n}\mathbf{B}_n \mathbf{s}_{k-1,\nu}(s)ds.\label{eq:s_nu_def2}
\end{align}
Lemma \ref{lem:coeff_s} below shows that for $k\ge 0$ the $\CC^n$ valued function $\mathbf{s}_{k,\nu}(u)$ is bounded in norm by a $k$-dependent constant multiple of $e^{\frac{\beta}{4} k u}$, and that these functions satisfy the ODEs
\begin{align}\label{eq:coeff_s_ode1}
     \mathbf{s}_{0,\nu}'(u)&= \mathbf{A}_n  \mathbf{s}_{0,\nu}(u)+ \frac{n+1}{2}\mathbf{e}_n,\\
        \mathbf{s}_{k,\nu}'(u)&= \mathbf{A}_n \mathbf{s}_{k,\nu}(u)+ i \tfrac{\beta}{4}e^{\frac{\beta}{4}u}\mathbf{B}_n \mathbf{s}_{k-1,\nu}(u), \qquad k\ge 1.\label{eq:coeff_s_ode2}
\end{align}
From this it follows that the series given by  \eqref{eq:q_nu_series} can be differentiated in $u$ term-by-term and the resulting series is absolutely convergent. Moreover, the series in \eqref{eq:q_nu_series} satisfies the ODE system \eqref{eq:q_nu_ODE} and hence it is the unique solution of that system.  Lemma \ref{lem:coeff_s} also shows that for any $u>-\infty$ we have $\lim\limits_{\nu\to -\infty} \mathbf{s}_{k,\nu}(u)=e^{k \frac{\beta}{4}u} \mathbf{s}_{k}$ where $\mathbf{s}_{k}$ satisfies \eqref{eq:q_coeff_1}, and that 
\begin{align}\label{eq:q_nu_lim}
    \lim_{\nu\to -\infty}  \tilde{\mathbf{q}}_{\nu}(u,\lambda)=\sum_{j=0}^\infty \mathbf{s}_j e^{\frac{\beta}{4} j u} \lambda^j.
\end{align}
From this, we obtain \eqref{eq:q_expansion}, and direct differentiation (which is allowed by Lemma \ref{lem:coeff_s}) gives \eqref{eq:ode_q_vec} as well.
\end{proof}

\begin{lemma}\label{lem:coeff_s} Consider the functions $ \mathbf{s}_{j,\nu}(u), j\ge 0, u\ge \nu$ defined in \eqref{eq:s_nu_def1}. For every $\nu\in \R$, these are well-defined, and they satisfy the  ODEs given in \eqref{eq:coeff_s_ode1} and   \eqref{eq:coeff_s_ode2}. Moreover, we have the following norm bounds for $\nu\le u$:
\begin{align}\label{eq:coeff_norm_bound}
        \|\mathbf{s}_{k,\nu}(u)\|\le e^{\frac{\beta}{4} k u} \,  \kappa^{k+1} \prod_{j=1}^k\frac{n}{j+\frac{4}{\beta}n},
\end{align}
with a constant $\kappa>0$ which only depends on $n$. 
Moreover, these functions satisfy 
\begin{align}\label{eq:coeff_nu_lim}
    \lim\limits_{\nu\to -\infty} \mathbf{s}_{k,\nu}(u)=e^{k \frac{\beta}{4}u} \mathbf{s}_{k}.
\end{align}
With $\tilde{\mathbf{q}}_\nu(u,\lambda)$ given by \eqref{eq:q_nu_series} we also have \eqref{eq:q_nu_lim}.
\end{lemma}
We will also need the following lemma.
\begin{lemma}\label{lem:An_eval}
The matrix $\mathbf{A}_n$ has  $n$ distinct  negative eigenvalues given by  $\gamma_{k,n}=\frac{k(k-1)}{2}-\frac{n(n+1)}{2}$, $1\le k\le n$. 
\end{lemma}
\begin{proof}
 Our starting point is the following version of the Chu-Vandermonde identity that holds for non-negative integers $a, b, r$ (see e.g.~(5.24) in \cite{Concrete}
).
 \begin{align}\label{eq:chu-v}
     \sum_{k=0}^a (-1)^k \binom{a}{k}\binom{k+r}{b}=(-1)^a \binom{r}{b-a}.
 \end{align}
Here $\binom{x}{y}=0$ if $y<0$ or $x<y$. Now consider the $n\times n$ matrix $\mathbf{T}_n$ defined as $[\mathbf{T}_n]_{a,b}=\binom{a}{b}(-1)^b$, note that this is a lower triangular matrix. Identity \eqref{eq:chu-v} with $r=0$ shows that $\mathbf{T}_n=\mathbf{T}_n^{-1}$. We claim that $\mathbf{T}_n \mathbf{A}_n \mathbf{T}_n$ is an upper bidiagonal matrix with entries
\begin{align}\label{eq:TAT}
 [\mathbf{T}_n \mathbf{A}_n \mathbf{T}_n]_{k,k}= \frac{k(k-1)}{2}-k n, \qquad [\mathbf{T}_n \mathbf{A}_n \mathbf{T}_n]_{k,k+1}=\frac{k(n-k)}{2}
\end{align}
 in the $k$th row. Since  $  \gamma_{n+1-k,n}=\frac{k(k-1)}{2}-k n$, from this the statement of the lemma follows.\\
To prove \eqref{eq:TAT} we start with the expression for the  $(\ell,m)$ entry of $\mathbf{T}_n\mathbf{A}_n\mathbf{T}_n$:
\begin{align}\label{eq:TAT1}
    [\mathbf{T}_n\mathbf{A}_n\mathbf{T}_n]_{\ell,m}=\sum_{j=1}^n t_{\ell,j} (a_{j,j} t_{j,m}+a_{j,j-1}t_{j-1,m}+a_{j,j+1}t_{j+1,m}). 
\end{align}
Here $t_{\ell,m}, a_{\ell,m}$ are the respective entries of $\mathbf{T}_n$ and $\mathbf{A}_n$, and they are defined to be zero if we do not have $1\le \ell, m\le n$. 
Using simple identities of binomial coefficients, we can rewrite each term in \eqref{eq:TAT1} as linear combinations of expressions of the form
$(-1)^j\binom{\ell}{j} \binom{j+r}{m}$ with finitely many choices for $r$. By applying  \eqref{eq:chu-v} to each of these terms, a somewhat lengthy but straightforward calculation shows that $\mathbf{T}_n\mathbf{A}_n\mathbf{T}_n$ is indeed upper bidiagonal with the claimed entries. 
\end{proof}

\begin{proof}[Proof of Lemma \ref{lem:coeff_s}]
The ODEs \eqref{eq:coeff_s_ode1} and   \eqref{eq:coeff_s_ode2} can be checked directly by differentiation from \eqref{eq:s_nu_def1} and \eqref{eq:s_nu_def2}.
 From Lemma \ref{lem:An_eval} it follows that the spectral radius of $e^{x \mathbf{A}_n}$ with $x\ge 0$ is $e^{-n x}$. This implies that there is a finite constant  $\kappa>0$ (depending only on $n$) so that
 \[
 \| e^{x \mathbf{A}_n} \mathbf{v}\|\le \kappa e^{-n x} \|\mathbf{v}\|.
 \]
The spectral radius of $\mathbf{B}_n$ is $n$, hence we get $\|\mathbf{B}_n \mathbf{v}\|\le n  \|\mathbf{v}\|$.
From \eqref{eq:s_nu_def1} and \eqref{eq:s_nu_def2}  we get
\begin{align}\label{eq:s_norm_bnd0}
    \|\mathbf{s}_{0,\nu}(u)\|&\le \kappa \frac{n+1}{2n}+\kappa e^{-(u-\nu)n} \|\mathbf{f}_n\|,\\
    \|\mathbf{s}_{k,\nu}(u)\|&\le \kappa n \frac{\beta}{4} e^{\frac{\beta}{4}u} \int_{\nu}^u e^{(s-u)(\frac{\beta}{4}+n)} \|\mathbf{s}_{k-1,\nu}(s)\| ds.
\end{align}
From \eqref{eq:s_norm_bnd0} it follows that $\|\mathbf{s}_{0,\nu}(u)\|$ is bounded by a constant uniformly in $\nu$, and by induction, it follows that there is a finite $\kappa_0>0$ so that 
\begin{align}\label{eq:s_nu_norm_bnd}
    \|\mathbf{s}_{k,\nu}(u)\|\le e^{\frac{\beta}{4} k u} \, \kappa_0 \kappa^k \prod_{j=1}^k\frac{n}{j+\frac{4}{\beta}n} .
\end{align}
Increasing the value of $\kappa$ gives the bound \eqref{eq:coeff_norm_bound}. 
Taking the limit $\nu\to -\infty$ in \eqref{eq:s_nu_def1} gives
\begin{align*}
  \lim_{\nu\to -\infty}   \mathbf{s}_{0,\nu}(u)&=\frac{n+1}{2}\int_{-\infty}^u e^{(u-s)\mathbf{A}_n} \mathbf{e}_nds+\lim_{\nu\to -\infty} e^{(u-\nu) \mathbf{A}_n }\mathbf{f}_n\\
  &=\frac{n+1}{2}\int_{0}^{\infty} e^{t \mathbf{A}_n} \mathbf{e}_ndt=-\frac{n+1}{2} \mathbf{A}_n^{-1} \mathbf{e}_n,
\end{align*}
proving the $k=0$ case of \eqref{eq:coeff_nu_lim}. 
Note that  
 $\mathbf{A}_n^{-1}$ is well-defined by Lemma \ref{lem:An_eval}. We can now prove  \eqref{eq:coeff_nu_lim} for general $k\ge 1$  by induction. Assume that \eqref{eq:coeff_nu_lim} holds for $k-1$ for some $k\ge 1$. 
By \eqref{eq:s_nu_norm_bnd} we can use dominated convergence to get
\begin{align*}
  \lim_{\nu\to -\infty}   \mathbf{s}_{k,\nu}(u)&=i \frac{\beta}{4}e^{\frac{\beta}{4} u}\int_{-\infty}^ue^{\beta/4 (s-u)}e^{(u-s)\mathbf{A}_n}\mathbf{B}_n e^{(k-1) \frac{\beta}{4}s}\mathbf{s}_{k-1}ds\\
  &=i \frac{\beta}{4}e^{\frac{k \beta}{4} u}\int_{-\infty}^ue^{k \frac{\beta}{4} (s-u)}e^{(u-s)\mathbf{A}_n}\mathbf{B}_n \mathbf{s}_{k-1}ds=i \frac{\beta}{4}e^{\frac{k \beta}{4} u}\int_{0}^{\infty} e^{(-k \frac{\beta}{4}\mathbf{I}+\mathbf{A}_n) t }dt \, \mathbf{B}_n\mathbf{s}_{k-1}
\\&=i e^{\frac{k \beta}{4} u}\left(k \mathbf{I}-\tfrac{4}{\beta}\mathbf{A}_n\right)^{-1} \, \mathbf{B}_n\mathbf{s}_{k-1}.
\end{align*}
Again, $\left(k \mathbf{I}-\tfrac{4}{\beta}\mathbf{A}_n\right)^{-1}$ is well-defined because of Lemma \ref{lem:An_eval}. This concludes the proof of \eqref{eq:coeff_nu_lim}. Finally, we use the norm bound \eqref{eq:coeff_norm_bound} again to conclude that 
\[
 \lim_{\nu\to -\infty} \tilde{\mathbf{q}}_{\nu}(u,\lambda)=
\sum_{j=0}^\infty \lim_{\nu\to -\infty} \mathbf{s}_{j,\nu}(u) \lambda^j
 =
 \sum_{j=0}^\infty \mathbf{s}_j e^{\frac{\beta}{4} j u} \lambda^j,
\]
proving \eqref{eq:q_nu_lim}.
\end{proof}
To obtain Corollary \ref{cor:int_case_rho} and Proposition \ref{prop:small_lambda_asym}, we prove the following claim. 

\begin{claim} We have the following identities:
   \label{cl:vAnBnf}
    \begin{align}
            &\mathbf{A}_n \mathbf{B}^{\ell}_n \mathbf{f}_n=\sum_{j=0}^{\lfloor \frac{\ell-1}{2}\rfloor} \left(\tbinom{\ell}{2j+2}-n\tbinom{\ell}{2j+1}\right)\mathbf{B}^{\ell-2j}_n\mathbf{f}_n, \quad \text{ for all $1\le \ell \le 2n$}
           \label{eq:AnBnf_2}
          \\
          \label{eq:vBnf}
          &\mathbf{v}_n\cdot \mathbf{B}^{2j}_n \mathbf{f}_n = 0,\;\; \text{ for all $1\le j<n$, \quad and }\qquad \mathbf{v}_n\cdot \mathbf{B}^{2n}_n \mathbf{f}_n = (-1)^n \binom{2n}{n}^{-1}\frac{(2n)!}{2}.
    \end{align}
    \end{claim}
\begin{proof}
From \eqref{eq:AB} and \eqref{eq:f} we have, 
\begin{align}\label{eq:AB1}
[\mathbf{B}^{\ell}_n \mathbf{f}_n]_k &= k^{\ell}, \quad
[\mathbf{A}_n \mathbf{f}_n]_k =-\frac{n+1}{2} \ind(k=1),
\qquad 1\le k\le n,\\
    [\mathbf{A}_n \mathbf{B}^\ell_n\mathbf{f}_n]_k&=\frac{k(k+n)}{2}(k-1)^\ell-k^{2+\ell}+\frac{k(k-n)}{2}(k+1)^\ell, \qquad 1\le k\le n.\label{eq:AB2}
\end{align}
Expanding $(k\pm 1)^\ell$ and checking the coefficients of $k$ yields \eqref{eq:AnBnf_2}. \\
To prove \eqref{eq:vBnf}, we start with the identity
    \begin{align*}
        \mathbf{v}_n\cdot \mathbf{B}^{2 j}_n \mathbf{f}_n  = \tbinom{2n}{n}^{-1} \sum_{k=1}^n (-1)^k \tbinom{2n}{n+k} k^{2j}  = \frac12 (-1)^n\tbinom{2n}{n}^{-1} \sum_{k=0}^{2n}(-1)^{2n-k} \binom{2n}{k} (k-n)^{2j}.
    \end{align*}
The sum on the right is just the result of the finite difference operator \[\Delta f(x)=f(x+1)-f(x)\] applied $2n$ times to the function  $f(x)=(x-n)^{2j}$, and then evaluated at $x=0$. 
Applying $\Delta^{2n}$ to the polynomial $x^{2n}$ yields the constant $(2n)!$, and applying the same operator to any polynomial with degree less than $2n$ yields the constant 0.  From this, we obtain \eqref{eq:vBnf}.
\end{proof}

\begin{proof}[Proof of Corollary \ref{cor:int_case_rho}]
The equation 
\[
\rho^{(1)}_{\beta,n}(\lambda) =\frac{1}{2\pi}\left(1+2\;\mathbf{v}_n \cdot \Re \mathbf{q}(\lambda)\right)
\]
follows from \eqref{eq:rho_HP_n} and the definitions. To prove the second part of \eqref{eq:int_case_rho}, we recall that we have $\mathbf{s}_0 = \mathbf{f}_n$. From $(1-1)^{2n} = 0$ we also get 
\[
\mathbf{v}_n \cdot \mathbf{s}_0 = \mathbf{v}_n \cdot \mathbf{f}_n = \sum_{k=1}^n (-1)^k \frac{\binom{2n}{n+k}}{\binom{2n}{n}} = -\frac{1}{2}.
\]
From the definition of $\mathbf{s}_j$ it follows that  $\mathbf{s}_j$ is purely imaginary when j is odd, and real when j is even. Hence we get
\begin{align*}
      &\rho^{(1)}_{\beta,n}(\lambda)  =\frac{1}{2\pi}\left(1+2\;\mathbf{v}_n \cdot \Re \mathbf{q}(\lambda)\right)=\frac{1}{2\pi}(1+ 2\sum_{j=0}^\infty \mathbf{v}_n\cdot \mathbf{s}_{2j} \lambda^{2j}) =\frac{1}{\pi} \sum_{j=1}^\infty \mathbf{v}_n\cdot \mathbf{s}_{2j} \lambda^{2j},
\end{align*}
which completes the proof. 
\end{proof}

\begin{proof}[Proof of Proposition \ref{prop:small_lambda_asym}]
    Following \eqref{eq:int_case_rho}, it suffices to show that
     \begin{align}
         \mathbf{v}_n\cdot\mathbf{s}_{2\ell}=0,\quad \text{for $1\le \ell\le n-1$},\quad\text{and} \quad 
         \mathbf{v}_n\cdot\mathbf{s}_{2n}=   \frac12 \binom{2n}{n}^{-1}\frac{\left(\tfrac{\beta}{2}\right)^{2n}}{\left(1+\frac{\beta}{2}\right)^{\uparrow 2n}}.
     \end{align}
Fix $1\le \ell\le 2n$. Consider the vector space $\mathcal{V}_\ell$ spanned by the vectors $\mathbf{B}_n^{\ell-2j}\mathbf{f}_n$, with $j=0,\dots,  \lfloor\frac{\ell-1}{2}\rfloor$.
Equation \eqref{eq:AnBnf_2} shows that $\mathbf{A}_n$ maps $\mathcal{V}_\ell$ to itself and acts as a lower-triangular matrix with respect to this basis. This implies that for $\gamma>0$ the mapping $\mathbf{I}-\gamma \mathbf{A}_n: \mathcal{V}_\ell \to \mathcal{V}_\ell$ is invertible, and the inverse is a lower triangular matrix with respect to our basis. 
In other words, there exist constants $ c_{\ell,\gamma, m, n} $ such that,
\begin{align}
    \label{eq:invBnf}
    (\mathbf{I}-\gamma \mathbf{A}_n)^{-1}   \mathbf{B}^{\ell}_n   \mathbf{f}_n = \sum_{m=0}^{\lfloor \frac{\ell-1}{2} \rfloor} c_{\ell,\gamma, m, n}  \mathbf{B}^{\ell-2m}_n \mathbf{f}_n.
\end{align}
Multiplying both sides by $(\mathbf{I}-\gamma \mathbf{A}_n)$ and matching the corresponding coefficients of $\mathbf{B}_n^{\ell} \mathbf{f}_n$ term, we get
\begin{align} \label{eq:invBnf_coeff}
 c_{\ell,\gamma, 0, n}=(1+ \gamma \ell (n-\tfrac{\ell-1}{2}))^{-1}.   
\end{align}
 By definition, we have for all $1\leq \ell \leq n,$ 
\begin{align*}
    \mathbf{s}_{2\ell} &= (-1)^{\ell} 
(2\ell \,\mathbf{I}-\tfrac{4}{\beta} \mathbf{A}_n)^{-1}\mathbf{B}_n ((2\ell-1) \mathbf{I}-\tfrac{4}{\beta} \mathbf{A}_n)^{-1}\mathbf{B}_n\cdots ( \mathbf{I}-\tfrac{4}{\beta} \mathbf{A}_n)^{-1}\mathbf{B}_n \mathbf{f}_n\\
 &= (-1)^{\ell} \frac{1}{(2\ell)!}
( \mathbf{I}-\tfrac{4}{\beta(2\ell)} \mathbf{A}_n)^{-1}\mathbf{B}_n (\mathbf{I}-\tfrac{4}{\beta (2\ell-1)} \mathbf{A}_n)^{-1}\mathbf{B}_n\cdots ( \mathbf{I}-\tfrac{4}{\beta} \mathbf{A}_n)^{-1}\mathbf{B}_n \mathbf{f}_n.
   \end{align*}
Applying  \eqref{eq:invBnf} repeatedly (from the right), we obtain that 
\[
  \mathbf{s}_{2\ell} =(-1)^\ell \frac{1}{(2\ell)!}\prod_{k=1}^{2\ell} c_{k,\frac{4}{\beta k },0,n} \,\mathbf{B}_n^{2\ell} \mathbf{f}_n + \sum_{k=1}^{\ell-1} \tilde{c}_{\ell, k, n} \, \mathbf{B}_n^{2\ell-2k} \mathbf{f}_n,
\]
with some finite coefficients  $\tilde{c}_{\ell, k, n}$.
Finally, by \eqref{eq:vBnf} and \eqref{eq:invBnf_coeff}, we conclude that
\begin{align*}
     \mathbf{v}_n\cdot \mathbf{s}_{2\ell} = 0 \quad \text{for } 1\le \ell\le n-1,
\end{align*}
and
\begin{align*}
    \mathbf{v}_n\cdot \mathbf{s}_{2n}
    = \frac{1}{2}\binom{2n}{n}^{-1} \prod_{j=1}^{2n}\left(1+\tfrac{4}{\beta}\Big(n-\tfrac{j-1}{2}\Big)\right)^{-1}.
\end{align*}
The last constant can also be written as 
\[
\frac12 \binom{2n}{n}^{-1}\frac{(\tfrac{\beta}{2})^{2n}}{\left(1+\frac{\beta}{2}\right)^{\uparrow 2n}},
\]
which completes the proof of Proposition \ref{prop:small_lambda_asym}.
\end{proof}

\subsection{Explicit solutions}\label{subs:explicit}

In the Introduction we have already shown how we can obtain the Sine-kernel, the pair correlation of the $\Sine_2$ process from Theorem \ref{thm:sineb_series}. In this section, we demonstrate how we can recover the known pair correlation function \eqref{eq:beta=4} for the $\operatorname{Sine}_4$ process, and also provide an explicit formula for the density of the $\HPb$ process for $\delta=1$.

When $n=2$ from \eqref{eq:AB} and \eqref{eq:vectors} we get 
\[
\mathbf{A}_2=\mat{-1}{-\frac12}{4}{-4}, \quad \mathbf{B}_2=\mat{1}{0}{0}{2}, \quad \mathbf{v}_2=[-\tfrac23, \tfrac16]^\dag.
\]
Hence, by Theorem \ref{thm:sineb_series} and Theorem \ref{thm:even_beta}, we have the following representation for the pair correlation of the  $\Sine_4$ process: 
\begin{align}\label{eq:rho4}
     \rho^{(2)}_{4}(0,\lambda)=\frac{1}{4\pi^2}+\frac1{2\pi^2}\left(-\frac{2}{3} \Re[q_1(\lambda)]+\frac16\Re[q_2(\lambda)]\right),
\end{align}
where $q_1, q_2$ solve the ODE system 
\begin{align}\label{eq:ODE4_1}
    \lambda q_1'(\lambda)&=\frac32-q_1(\lambda)-\frac12 q_2(\lambda)+i \lambda q_1(\lambda),\qquad q_1(0)=1,\\
    \lambda q_2'(\lambda)&=4 q_1(\lambda)-4 q_2(\lambda)+2 i \lambda q_2(\lambda), \qquad \phantom{....}q_2(0)=1.\label{eq:ODE4_2}
\end{align}
This system can be reduced to the following second-order ODE for $q_2$:
\begin{align*}
    \lambda^2 q_2''(\lambda)&=\left(-6 \lambda +3 i \lambda ^2\right) q_2'(\lambda )+\left(2 \lambda ^2+8 i \lambda -6\right) q_2(\lambda )+6,\qquad q_2(0)=1, q_2'(0)=\frac{2 i}{3},
\end{align*}
which  has the unique solution
\[
q_2(\lambda)=
\frac{3i(1+2i\lambda-e^{2i \lambda})}{2\lambda^3}-3i e^{i \lambda}\lambda^{-2}\int_0^\lambda \frac{\sin(x)}{x}dx.
\]
We can now obtain $q_1(\lambda)$ from \eqref{eq:ODE4_2}, 
and substituting these into  \eqref{eq:rho4} yields
the classical formula  \eqref{eq:beta=4} for  $\rho^{(2)}_{4}(0,\lambda)$.
\medskip

To obtain the density of the $\operatorname{HP}_{\beta,1}$ process we can use Theorem \ref{thm:even_beta} or Corollary \ref{cor:int_case_rho}. Here $n=1$, and we have $\mathbf{A}_1=-1$, $\mathbf{B}_1=1$, $\mathbf{v}_1=-\tfrac{1}{2}$.  Hence  $\mathbf{s}_k$ satisfies the scalar recursion 
\[
\mathbf{s}_0=1, \qquad \mathbf{s}_k=i (k+\tfrac{4}{\beta})^{-1} \mathbf{s}_{k-1}, \qquad k\ge 1
\]
which yields $\mathbf{s}_k=\frac{i^k}{\left(1+\tfrac{4}{\beta}\right)^{\uparrow k}}$.  Corollary \ref{cor:int_case_rho} gives
\begin{align}
   \rho_{\beta,1}^{(1)}(\lambda)=-\frac1{2\pi} \sum_{k=1}^\infty \frac{(-1)^k }{\left(1+\tfrac{4}{\beta}\right)^{\uparrow 2k}}
   \lambda^{2k}
   =
   \frac{1 }{4 \pi (1+ \frac{2}{\beta}) (1+\frac{4}{\beta})} \, _1F_2\left(1;\frac{3}{2}+\frac{2}{\beta },2+\frac{2}{\beta };-\frac{\lambda ^2}{4}\right) \lambda ^2.
\end{align}
This expression can also be obtained by solving the ODE given by  Theorem \ref{thm:even_beta}:
\begin{align*}
    \tfrac{\beta}{4}\lambda q'(\lambda) &= (i\tfrac{\beta}{4} \lambda -1) q(\lambda) + 1 , \quad q(0) = 1.\\
    q(\lambda)
&= \frac{4}{\beta}\,e^{i\lambda}\,\lambda^{-4/\beta}
\int_{0}^{\lambda} s^{\,4/\beta - 1}\,e^{-i s}\,ds, \qquad \text{for }\lambda>0.
\end{align*}
The first equation in Corollary \ref{cor:int_case_rho} now yields the equivalent expression
\begin{align*}
     \rho_{\beta,1}^{(1)}(\lambda) = \frac{1}{2\pi}-\frac{1}{2\pi}\Re{q(\lambda)}=
\frac{1}{2\pi}-\frac{2}{\beta \pi}\lambda^{-4/\beta} \int_0^\lambda s^{4/\beta-1} \cos(\lambda-s) ds,
\end{align*}
for $\lambda>0$.

\subsection{ODE representation for $\rho_{2n}^{(2)}(0,\lambda)$} \label{subs:ODE_repr}

Theorem \ref{thm:even_beta} and Corollary \ref{cor:int_case_rho} with $\beta=2n$ provide a description for $\rho_{2n}^{(2)}(0,\lambda)$ as 
\[
\rho_{2n}^{(2)}(0,\lambda)=\frac{1}{4\pi^2}\left(1+2\;\mathbf{v}_n \cdot \Re \mathbf{q}(\lambda)\right)
\]
with $\mathbf{q}$ as the solution of the ODE \eqref{eq:ode_q_vec}. Motivated by \cite{Forrester_20212}, one can ask if there is a way to represent $\rho_{2n}^{(2)}(0,\lambda)$ itself as the solution of a differential equation using our representation. 
We outline a method that produces an order $2n$ real ODE with polynomial coefficients, without spelling out the computational details. 

As a first step, we show that the function $r_n(\lambda)=2 \mathbf{v}_n^T \mathbf{q}(\lambda)$ solves an order $n$ complex-valued linear differential equation where the coefficients are polynomials with degree at most $n$. This gives a representation for  $\rho_{2n}^{(2)}(0,\lambda)$ as the \emph{real part} of the solution of an order $n$  linear differential equation.

It can be shown that $\mathbf{v}_n^\dag$ is the left eigenvector of $\mathbf{A}_n$ corresponding to the eigenvalue $-n$. From Lemma \ref{lem:An_eval} it follows that  $\mathbf{A}_n$ has $n$ linearly independent left eigenvectors, let $\mathbf{w}_k^\dag, 1\le k\le n-1$, denote the  ones additional to $\mathbf{v}_n^\dag$. Introducing the new variables $r_k(\lambda)=\mathbf{w}_k\cdot  \mathbf{q}(\lambda)$, $1\le k\le n-1$, we see that $\mathbf{r}(\lambda)=[r_1(\lambda),\dots,r_n(\lambda)]^\dag$ satisfies a transformed version of \eqref{eq:ode_q_vec}, with $\mathbf{B}_n, \mathbf{A}_n, \mathbf{e}_n, \mathbf{f}_n$ replaced with their respective representations in a new basis. In this new representation, $\mathbf{A}_n$ will become a diagonal matrix, and by choosing the normalization of the eigenvectors $\mathbf{w}_k, 1\le k\le n-1$, appropriately, one can show that $\mathbf{B}_n$ will turn into a tridiagonal matrix with explicitly computable entries. This allows us to successively eliminate $r_{n-1}, r_{n-2}, \dots, r_1$ from the ODE system for $\mathbf{r}(\lambda)$, leading to an order $n$ linear ODE  with polynomial coefficients of degree at most $n$. The initial condition for $\mathbf{r}(0)$ will be transformed into equations for $r_n(0), r'_n(0), \dots, r_n^{(n-1)}(0)$.

Once we have an order $n$ complex-valued ODE for $r_n$, we can take real and imaginary parts to get a system of two differential equations for $\Re r_n$ and $\Im r_n$. Successive elimination of the highest order derivatives of $\Im r_n$ from these two equations yields an order $2n$ linear real ODE for $\Re r_n$, with polynomials of degree at most $2n$ as coefficients. Since $\Re r_n(\lambda)=4\pi^2 \rho_{2n}^{(2)}(0,\lambda)-1$, a simple transformation of the resulting equation gives an order $2n$ linear differential equation for $ \rho_{2n}^{(2)}(0,\lambda)$.
For small $n$, the procedure can be easily carried out with the help of computer algebra. For example, here is the 6th order linear differential equation satisfied by $ \rho_6(\lambda):=\rho_{6}^{(2)}(0,\lambda)$:
\begin{align*}
\frac{81\lambda^6}{\pi^2} =\; & 9\lambda^6\,\partial_\lambda^6 \rho_6(\lambda) 
+ 168\lambda^5\,\partial_\lambda^5 \rho_6(\lambda) + (126\lambda^6 + 868\lambda^4)\,\partial_\lambda^4 \rho_6(\lambda) \\
& 
+ (1176\lambda^5 + 1040\lambda^3)\,\partial_\lambda^3 \rho_6(\lambda)  + (441\lambda^6 + 2196\lambda^4 -804\lambda^2)\,\partial_\lambda^2 \rho_6(\lambda) \\
& + (1296\lambda^5 -144\lambda^3 -528\lambda)\,\partial_\lambda \rho_6(\lambda)+ (324\lambda^6 + 216\lambda^4 -396\lambda^2 +168)\,\rho_6(\lambda).
\end{align*}
We note that the recent paper \cite{betasix2026} (building on the results of this section) showed that one can perform a further reduction in the resulting ODE. When $\beta=2n$, the vector-valued function $\mathbf{q}(\lambda)$ can be expressed in terms of the solution of an order $n-1$ complex valued differential equation. This was a key step in \cite{betasix2026}, where it was used to derive a representation of $\rho_{6}^{(2)}(0,\lambda)$ in terms of Bessel functions.

\section{Appendix}\label{sec:Appendix}

\subsection{The random operator representation of $\HPb$}\label{subs:operator_overview}

We start with a quick review of the operator theoretic framework developed to study the $\HPb$ process based on the results of \cite{BVBV_sbo}, \cite{LV_HP}, \cite{BVBV_Palm}. As in the rest of the paper, we will assume that the $\delta$ parameter is real and non-negative, most of the results below extend to the general $\Re \delta>-\frac12$ case. 

Let $B_1, B_2$ be independent two-sided standard Brownian motions on $\R$. Define 
\begin{align}\label{eq:xy}
y(u)=e^{B_2(u)-(\delta+\tfrac12)u}, \qquad     x(u)=\begin{cases}
-\int_u^0 e^{B_2(s)-(\delta+\tfrac12)s} d B_1,\qquad & u\le 0,\\
\;\;\;\,\int_0^u e^{B_2(s)-(\delta+\tfrac12)s} d B_1 &u\ge0.
\end{cases}
\end{align}
Using the law of iterated logarithm and the Dubins-Schwarz theorem, it can be shown that for any $\eps>0$ there is a random $C_{\beta, \delta, \eps}$ so that the following bounds hold a.s.
\begin{align}\label{eq:xy_bounds}
C_{\beta, \delta, \eps}^{-1} e^{-(\delta+\tfrac12-\eps)u}\le  y(u)\le C_{\beta, \delta, \eps}  e^{-(\delta+\tfrac12+\eps)u}, \quad |x(u)| y(u)^{-1}\le C_{\beta, \delta, \eps} e^{-\eps u}, \qquad u\le 0. 
\end{align}
Consider the matrix-valued process 
\begin{equation}
\label{eq:X}
X(u)=\mat{1}{-x(u)}{0}{y(u)}, \qquad u\in \R,
\end{equation}
and introduce the time change 
\begin{equation}\label{eq:time_chng}
u_\beta(t)=\tfrac{4}{\beta}\log t.
\end{equation} 
Set 
\begin{align}\label{eq:R}
    R(t)=\frac{1}{2 \det X(u_\beta(t))}  X(u_\beta(t))^{\dag}X(u_\beta(t)), \qquad t\in (0,1],
\end{align}
where $X^\dag$ is the transpose of $X$. Let $q$ be a random variable independent of $B_1, B_2$ with probability density function proportional to $(1+x^2)^{-\delta-1}$. It can be checked that $q\ed -\cot(\Theta/2)$ with $\Theta$ defined with the density function \eqref{eq:Theta}. The following theorem is proved in \cite{LV_HP}, for the case $\delta=0$, it was shown previously in \cite{BVBV_szeta}.
\begin{theorem}\label{thm:HP_operator}
 Consider the differential operator
 \begin{align}\label{eq:Dirac}
     \tau^{\ttHP}_{\beta,\delta}f(t)=R(t)^{-1}J f'(t), \qquad J=\mat{0}{-1}{1}{0}, 
 \end{align}
acting on functions $f:(0,1]\to \CC^2$, with left and right boundary conditions $\uu_0=[1,0]^{\dag}, \uu_1=[-q,-1]^{\dag}$. The operator $\tau^{\ttHP}_{\beta,\delta}$ is self-adjoint on an appropriate domain, and its spectrum has the same distribution as the $\HPb$ process. 
\end{theorem}
For a precise description of the domain see (12) of \cite{BVBV_szeta}. The operator $\tau^{\ttHP}_{\beta,\delta}$ is a member of the Dirac-type operators studied in \cite{BVBV_szeta}. In particular, using the bounds \eqref{eq:xy_bounds}, it can be shown that $\tau^{\ttHP}_{\beta,\delta}$ with the boundary conditions $\uu_0, \uu_1$ satisfies Assumptions 1, 2 and 3 of that paper. As a consequence of the results of \cite{BVBV_szeta}, the spectrum of $   \tau^{\ttHP}_{\beta,\delta}$ can be characterized using the solutions of the following family of differential equations (a canonical system): 
\begin{align} \label{eq:H}
    J H'(t,z)=z R(t) H(t,z), \qquad \lim_{t\to 0^{+}} H(t,z)=\uu_0.
\end{align}
This system (formally) is the eigenvalue equation of the operator $\tau^{\ttHP}_{\beta,\delta}$.
Proposition 18 of \cite{BVBV_szeta} shows that there is a unique solution $H:(0,1]\times \CC\to \CC^2$ of this system which is analytic in $z$,  it is never equal to $[0,0]^\dag$, and satisfies $H(t,0)=\uu_0$. Moreover, the function 
\begin{align}
    \zeta_{\beta, \delta}^{\ttHP}(z)=H(1,z)^{\T}J\uu_1
\end{align}
is an entire function that is real for $z\in \R$, and has a zero set that is given by the spectrum of $\tau^{\ttHP}_{\beta,\delta}$. In words, this describes the spectrum as the result of a shooting problem: $\lambda$ is an eigenvalue exactly if $H(1,\lambda)\parallel \uu_1$.

Proposition 26 of \cite{BVBV_szeta} shows that the solution $H$ of \eqref{eq:H} can be approximated  uniformly on compacts in $z$ with the functions $H_\eps(t,z)$ which are the solutions of the same differential equation \eqref{eq:H} on $[\eps,1]$ with initial condition $H_\eps(\eps, z)=\uu_0$. 

Consider now the vector-valued process $\mathcal H(t,z)=X(u_\beta(t)) H(t,z):=[\cA(t,z),\cB(t,z)]^{\dag}$, from the definitions we have $\mathcal H(1,z)=H(1,z)$.  Moreover, by the shooting problem description above, $\lambda\in \R$ is an eigenvalue of $\tau^{\ttHP}_{\beta,\delta}$ exactly if $\frac{\cA(1,\lambda)}{\cB(1,\lambda)}=q$. Instead of following a ratio, we can follow an angle.  For $\lambda\in\R$ we define the real valued function $\psi(t, \lambda)$ via
\begin{align}
  \psi(t,\lambda)=2 \Im \log(\cA(t,\lambda)-i \cB(t,\lambda)).
\end{align}
Since $\cH(t,z)$ is continuous in $z$ and is never equal to $[0,0]^\dag$,  we can choose a continuous branch of the logarithm with $\psi(t,0)=0$.

Theorem 24 of \cite{LV_HP} shows that the process $t\mapsto  \mathcal H(t,\lambda)$ solves a stochastic differential equation which in turn yields a stochastic differential equation for $\psi(t,\lambda)$. Let $\alpha_\lambda(u)=\psi(\exp(\tfrac{\beta}{4} u),\lambda), u\le 0$ be the process we obtain from $\psi(t,\lambda)$ by applying the inverse of the 
 time-change \eqref{eq:time_chng}. Theorem 20 of \cite{LV_HP} shows that  $\alpha_\lambda(u)$ satisfies the SDE  \eqref{eq:alphaSDE} with the appropriate boundary condition. The characterization of $\HPb$ given in \eqref{eq:HP} is the result of translating the shooting problem above, together with the distributional identity between $q$ and $\Theta$. 

The statements of Proposition \ref{prop:alphaSDE} now follow easily from the monotonicity and continuity properties of the function $H(t,z)$ of \eqref{eq:H}, together with the fact that $\alpha_\lambda$ is just a time-changed version of the process $2\Im \log\left([1,-i] X(u_\beta(t)) H(t,\lambda)\right)$.

\subsection{$\HPbb$ is the reduced Palm measure of $\Sineb$}\label{subs:PalmSine}

In this section we prove Proposition \ref{prop:Palm}, that $\HPbb$ is the reduced Palm measure of the $\Sineb$ process. All the elements of the proof are contained in \cite{BVBV_Palm} and \cite{LV_HP}, we just have to collect and organize them. 

\begin{proof}[Proof of Proposition \ref{prop:Palm}]
Consider the operator $\btau_\beta^{\mathtt{Sine}}=\tau^{\ttHP}_{\beta,0}$. The spectrum of this operator is the $\Sineb$ process (see \cite{BVBV_szeta}, \cite{BVBV_Palm}). Theorem 1 of \cite{BVBV_Palm} proves that the Palm measure of $\Sineb$  has the same distribution as the $\btau_\beta^{\mathtt{Sine}}$ operator with modified right boundary condition  $\uu_1=[1,0]^{\dag}$. Using the canonical system \eqref{eq:H} corresponding to  $\btau_\beta^{\mathtt{Sine}}$, and denoting its solution by $H^{\mathtt{Sine}}(t,z)$,  we obtain that the Palm measure of $\Sineb$ is the zero set of the random entire function $[0,1]\cdot H^{\mathtt{Sine}}(1,\lambda)$. 

Theorem 1 of \cite{BVBV_Palm} relies on the fact that in an appropriate coupling, the  operator $\btau_\beta^{\mathtt{Sine}}$ and the  function $H^{\mathtt{Sine}}(t,z)$ can  be approximated with similar objects corresponding to the size $n$ circular beta-ensemble. The fact that one can construct Dirac-type operators corresponding to these finite ensembles was first proved in \cite{BVBV_sbo} using the results of Killip-Nenciu \cite{KillipNenciu}. 
In \cite{BVBV_op} it was shown that in a certain coupling the  operators corresponding to the finite ensembles approximate 
$\btau_\beta^{\mathtt{Sine}}$ (or more precisely: a conjugated version of this operator) almost surely in the strong resolvent sense. Section 5 of \cite{BVBV_Palm} shows that in this coupling the solutions of the appropriate canonical systems converge a.s.~uniformly on compacts. 

An intermediate step of the proof of Theorem 1 in \cite{BVBV_Palm} is to show that the Palm measure of $\Sineb$ is the point process scaling limit of the Palm measure of the size $n$ circular beta-ensemble. (See the proof of Proposition 3 and the discussion in Section 5 of \cite{BVBV_Palm}.) On the other hand, Propositions 5 and 6 of \cite{BVBV_Palm} imply that the Palm measure of the 
size $n$ circular beta-ensemble can be obtained by choosing one of its $n$ points uniformly at random and rotating the whole ensemble so that this point moves to $z=1$. A direct computation (using the rotational invariance of the circular beta-ensemble) shows that this is just the size $n-1$ circular Jacobi beta-ensemble with $\delta=\beta/2$, with an additional point at $z=1$. 
As \cite{LV_HP} shows, the point process scaling limit (as $n\to \infty$) of these ensembles is exactly the $\HPbb$ process with an extra point at $0$. This shows that the reduced Palm measure of $\Sineb$ is given by the $\HPbb$ process, completing the proof of Proposition \ref{prop:Palm}.
\end{proof}

\subsection{Proof of Proposition \ref{prop:sde_cont}}\label{subs:continuity}

We provide a proof of Proposition \ref{prop:sde_cont}, showing that the $\alpha_\lambda(u)$ process in \eqref{eq:alphaSDE_b} can be realized jointly for all $\beta>0$, so that the resulting process is continuous in $\beta$.

\begin{proof}[Proof of Proposition \ref{prop:sde_cont}]

Fix $B_1, B_2$ independent two-sided standard Brownian motion on $\R$, and define $x_\beta(u), y_\beta(u), u\in \R$ using \eqref{eq:xy} with $\delta=\beta/2$. We will only use these processes on $(-\infty,0]$. Note that the reversed process $\tilde z_\beta(v)=x_\beta(-v)+i y_\beta(-v), v\ge 0$ satisfies the SDE 
\begin{align}\label{eq:zzz}
d\tilde z_\beta(v)=\Im \tilde z_\beta(v) (dW+i (\tfrac{\beta}{2}+1) dv), \qquad \tilde z_\beta(0)=i
\end{align}
with $W(v)$ being the standard complex Brownian motion $B_1(-v)+i B_2(-v)$.

Recall the time change $u_\beta$ defined in \eqref{eq:time_chng}. It is clear from \eqref{eq:xy} (or by a standard application of the Kolmogorov continuity theorem) that we can realize 
$x_\beta(u_\beta(t)), y_\beta(u_\beta(t))$ for $t\in (0,1], \beta>0$ so that these processes are jointly continuous in $t, \beta$. Now consider  $X_\beta$ and $R_\beta$ defined in \eqref{eq:X} and \eqref{eq:R} (using $x_\beta, y_\beta$), and the solution $H_\beta$ of the canonical system \eqref{eq:H}. As we have seen in Section \ref{subs:operator_overview}, the process $\alpha_{\lambda, \beta}(u), u\le 0$ can be obtained by turning the complex number $[1,-i]X_\beta(u_\beta(t)) H_\beta(t,\lambda)$ into an angle (by using a continuous lifting of the $\arg$ function) and applying the inverse time change $t_\beta(u)=\exp(\tfrac{\beta}{4}u)$. Hence, in order to show that $\alpha_{\lambda,\beta}(u)$ is continuous in $\beta$, it is sufficient to prove that $H_\beta(t,z)$ is continuous in $\beta$ for a fixed $t\in (0,1]$. 

We recall from \cite{BVBV_szeta}, \cite{BVBV_Palm} some additional properties of canonical systems of the form \eqref{eq:H}. Let $\hat x(t)+i \hat y(t), t\in (0,1]$ be a path in the open upper half plane ($\hat y(t)>0$). Set $\hat X(t)=\mat{1}{-\hat x(t)}{0}{\hat y(t)}$ and $\hat R(t)=\frac{1}{2\det \hat X(t)} \hat X(t)^\dag \hat X(t)$. Assume that $\|\hat R\|, \|\hat R^{-1}\|$ are both locally bounded in $(0,1]$, and that the following integrals are finite:
\begin{align}\label{eq:first_cond}
    \int_0^1 \hat y(t)^{-1}dt <\infty \qquad \int_0^1\left|\tfrac{\hat x(t)}{\hat y(t)}\right| dt<\infty \qquad \int_0^1 \int_0^t \frac{\hat x(t)^2+\hat y(t)^2}{\hat y(t) \hat y(s)} ds dt<\infty.
\end{align}
Under these assumptions, there is a unique solution of the canonical system
\begin{align} \label{eq:Hhat}
    J \hat H'(t,z)=z \hat R(t) \hat H(t,z), \qquad \lim_{t\to 0^{+}} \hat H(t,z)=[1,0]^\dag.
\end{align}
The solution $\hat H$ is an entire function in $z$, satisfies $\hat H(t,0)=[1,0]^\dag$, and it is never equal to $[0,0]^\dag$. (See Proposition 18 of \cite{BVBV_szeta} and Proposition 11 of \cite{BVBV_Palm}.)

Proposition 12 of \cite{BVBV_Palm} provides an explicit bound for $\sup_{t\in (0,1], |z|\le K}|\hat H_1(t,z)-\hat H_2(t,z)|$ where $\hat H_1, \hat H_2$ are solutions of  canonical systems generated from  paths $\hat x_j(t)+i \hat y_j(t)$, $j=1,2$ satisfying these conditions. 
We state a rephrased version here, which is a simple consequence of that  proposition. 
Suppose that $\hat x_j(t)+i \hat y_j(t)$, $j\in \Z_{\ge 0}$ are paths satisfying the conditions listed above. Assume also that the following integrals all go to 0 as $n\to \infty$:
\begin{align}\label{eq:cond_second}
  \int_0^1 |\hat y_n(t)^{-1}-\hat y_0(t)^{-1}|dt, \qquad  & \int_0^1 |\hat y_n(t)^{-1/2}-\hat y_0(t)^{-1/2}|^2dt,\qquad \int_0^1 \left|\tfrac{\hat x_n(t)}{\hat y_n(t)}-\tfrac{\hat x_0(t)}{\hat y_0(t)}\right|dt, \\[5pt]& \int_0^1 \int_0^t \left\|\tfrac{[\hat x_n(t),-\hat y_n(t)]}{\sqrt{\hat y_n(t) \hat y_n(s)}}-\tfrac{[\hat x_0(t), -\hat y_0(t)]}{\sqrt{\hat y_0(t) \hat y_0(s)}}\right\|^2 ds dt.
\end{align}
Let $\hat X_n, \hat R_n$ and $\hat H_n$ be the respective objects defined from $\hat x_n, \hat y_n$ (see \eqref{eq:Hhat} and the discussion before it).  Then for any $K>0$ we have
\begin{align*}
   \lim_{n\to \infty}  \sup_{t\in (0,1], |z|\le K}|\hat H_n(t,z)-\hat H_0(t,z)|=0.
\end{align*}
The fact that \eqref{eq:first_cond} holds for $\hat x(t)=x_\beta(u_\beta(t))$, $\hat y(t)=y_\beta(u_\beta(t))$ for a fixed $\beta>0$ follows from the a.s.~bounds of \eqref{eq:xy_bounds} (with $\delta=\beta/2$). To show the continuity of $H_\beta(t,z)$ in $\beta$ we just have to show that 
 as $\beta_n\to \beta_0>0$ the four integrals in \eqref{eq:cond_second} (with $\hat x_j(t)=x_{\beta_j}(u_{\beta_j}(t))$, $\hat y_j(t)=y_{\beta_j}(u_{\beta_j}(t))$) all go to 0. This follows from the dominated convergence theorem if we can show a version of the  bounds in \eqref{eq:xy_bounds} (with $\delta=\beta/2$) that is locally uniform in $\beta$. 

From the law of iterated logarithm, for any $\eps>0$ there is a random $C_\eps$ so that 
\begin{align}\label{eq:xy_bounds_2}
C_{\eps}^{-1} e^{-(\beta/2+\tfrac12-\eps)u}\le  y_\beta(u)\le C_{\eps}  e^{-(\beta/2+\tfrac12+\eps)u}, 
\end{align}
holds  uniformly in $\beta$. Set $r_\beta(v)=\frac{x_\beta(-v)}{y_\beta(-v)}$. By \eqref{eq:xy_bounds}, for a fixed $\beta_0>0$ we can find a random $C_{1}$ so that
\begin{align}\label{eq:another_bound}
    |r_{\beta_0}(v)|\le C_1 e^{\eps v}.
\end{align}
Introduce $\tilde r_{\beta, \beta_0}(v):=r_{\beta}(v)-r_{\beta_0}(v)$.

Using \eqref{eq:zzz} with $W=W_1+i W_2$, It\^{o}'s formula gives
\begin{align*}
dr_\beta(v)&=d W_1-r_\beta(v) dW_2-\tfrac{\beta}{2} r_\beta(v) dv, \qquad\qquad r_\beta(0)=0
\\
    d\tilde r_{\beta, \beta_0}(v)&=-\tilde r_{\beta, \beta_0}(v) (dW_2 +\tfrac\beta2 dv)-\tfrac12 (\beta-\beta_0) r_{\beta_0}(v) dv, \qquad \tilde r_{\beta, \beta_0}(0)=0. 
\end{align*}
We can solve the second equation using the integrating factor method to obtain 
\begin{align*}
    \tilde r_{\beta, \beta_0}(v)=-\tfrac{1}{2}(\beta-\beta_0)
e^{-W_2(v)-\frac{\beta+1}{2}v}
\int_0^v
e^{W_2(s)+\frac{\beta+1}{2}s}
r_{\beta_0}(s)\,ds.
\end{align*}
Using \eqref{eq:another_bound} and the sublinear growth of $W_2(v)$ we can now find a random constant $C_2$ so that 
\begin{align*}
    |\tilde r_{\beta, \beta_0}(v)|\le C_2 |\beta-\beta_0|  e^{2\eps v},
\end{align*}
which (using \eqref{eq:another_bound} again) implies 
\[
|x_\beta(-v)| y_\beta^{-1}(-v)=|r_\beta(v)|\le |\tilde r_{\beta, \beta_0}(v)|+|r_{\beta_0}(v)|\le (C_1+C_2 |\beta-\beta_0|)e^{2\eps v}.
\]
This yields the required locally uniform in $\beta$ version of the bound \eqref{eq:xy_bounds}, which is what we needed to finish our proof. 
\end{proof}

\subsection{Existence of the $k$-point correlation function for $\Sineb$}\label{subs:kcorr_exist}

\begin{theorem}\label{thm:corr_existence}
Fix $\beta>0$. For any $k\ge 1$, the $k$-point correlation function of $\Sineb$ is well-defined.  
\end{theorem}

One of the main tools of the proof of Theorem \ref{thm:corr_existence} is the  Dobrushin-Lanford-Ruelle characterization of the $\Sineb$ process given in \cite{DLR2021} by Dereudre, Hardy, Lebl\'e, and Ma\"ida.
The following theorem summarizes the parts of  Theorem 1.1 and Corollary 1.2 of \cite{DLR2021} that we used in our proof.

\begin{theorem}[DLR description of $\Sineb$, \cite{DLR2021}]\label{thm:DLR}
Fix $\beta>0$ and suppose that the point process $\Xi$ has $\Sineb$ distribution. For a subset $\Lambda\subset \R$ we denote by $\operatorname{Conf}(\Lambda)$ the set of configurations of locally finite point processes on $\Lambda$, and by $\Xi|_{\Lambda}$ the restriction of $\Xi$ to $\Lambda$.
Let $J\subset \R$ be a compact interval. Then there exists a measurable function 
\[
\operatorname{Number}_{J} :\operatorname{Conf}(J^c)\to \Z_{\ge 0}\]
so that $N_{\Xi}(J)=\operatorname{Number}_{J}(\Xi|_{J^c})$ with probability one. Moreover,  the law of $\Xi|_J$ given
the exterior configuration $\Xi|_{J^c}$ is mutually absolutely continuous with respect to the law of the point configuration formed by $n$ independent uniformly chosen points from $J$,
 where $n=\operatorname{Number}_{J}(\Xi|_{J^c})$ is the number of points in $J$. If $n=0$ then reference law is the point mass at the empty configuration. 
\end{theorem}

\begin{proof}[Proof of Theorem \ref{thm:corr_existence}]
Let $\R_{<}^k$ denote the open Weyl-chamber $\{(x_1,\dots, x_k): x_1<x_2<\dots<x_k\}$ with $\R_{<}^1=\R$. 
Suppose that the simple point process $\Xi$ has $\Sineb$ distribution. In this section, we denote by $N_{\Xi}$ the counting measure of $\Xi$, for a Borel set $A\subset \R$ we have
\begin{align}
    N_{\Xi}(A)=\sum_{\lambda\in \Xi} \ind(\lambda\in A).
\end{align}
For $k\ge 1$ define $\Xi^{(k)}$ as the point process on $\R_{<}^k$ built from distinct ordered $k$-tuples of $\Xi$. In other words, $(\lambda_1,\dots, \lambda_k)\in \Xi^{(k)}$ exactly if $\lambda_1<\dots<\lambda_k$ are all points of $\Xi$. Let $N_{\Xi}^{(k)}$ be the counting measure of $\Xi^{(k)}$. 
The expected value of this measure is the $k$th factorial moment measure of $\Xi$ on $\R_{<}^k$, which we denote as 
\begin{align}
    M_{\Xi}^{(k)}(A)=E\left[\sum_{x\in \Xi^{(k)}} \ind(x\in A)\right], \qquad A\subset \R_{<}^k.
\end{align}
For $k=1$, this is just the expected counting measure, in this case we use the notation $M_{\Xi}$.

To prove the theorem, it is sufficient to show that for each $k\ge 1$ the measure $M_{\Xi}^{(k)}$ is locally finite and it is absolutely continuous with respect to the Lebesgue measure on $\R_{<}^k$.

Suppose that $A\subset \R_{<}^k$ is a bounded Borel set. Choose a compact interval $J$ with 
 $A\subset J^k$. The following bound is a consequence of  the definitions:  
\begin{align}
     N_{\Xi}^{(k)}(A)\le \frac{1}{k!} N_\Xi(J)^k.
\end{align}
 Proposition 9 of \cite{BVBV} provides an exponential tail bound on $N_\Xi(J)$ when $\Xi\sim \Sineb$. This implies that $M_{\Xi}^{(k)}(A)=E[N_{\Xi}^{(k)}(A)]$ is finite.

To finish the proof, we use Theorem \ref{thm:DLR} to show  that $M_{\Xi}^{(k)}$ is absolutely continuous with respect to the Lebesgue measure on $\R_{<}^k$. 
Assume that $A\subset \R_{<}^k$ is a  Borel set with Lebesgue measure equal to zero. We also assume that $A\subset J^k$ with a compact interval $J$.  We claim that the conditional expectation of the  random variable $N_\Xi^{(k)}(A)$ with respect to the configuration on $J^c$ is zero. By Theorem \ref{thm:DLR}, $N_{\Xi}(J)$ is a measurable function of  the restricted point process $\Xi|_{J^c}$. If $n=N_{\Xi}(J)$ is 0 then by definition $N_\Xi^{(k)}(A)=0$. When $n\ge 1$, the  conditional law of $\Xi|_J$ given $\Xi|_{J^c}$ is absolutely continuous with respect to the distribution of $n$ unlabeled i.i.d.~points chosen  uniformly at random in $J$. Hence there exists a measurable function $f$ (depending on $n$, $J$ and the configuration $\Xi|_{J^c})$ so that the conditional expectation of $N_\Xi^{(k)}(A)$ is given by 
\[
\int_{J^n} \sum_{1\le i_1,\dots, i_k\le n \text{ distinct}} \ind((x_{i_1},\dots, x_{i_k})\in A) f(x_1,\dots, x_n) dx_1 \cdots dx_n.
\]
Since $A$ has measure zero, each of the finitely many terms in the sum integrate to zero, which makes the conditional expectation 0 as well. This implies that $M_\Xi^{(k)}(A)=E[N_\Xi^{(k)}(A)]$ also must be zero. 
This is sufficient to show the absolute continuity of the measure $M_\Xi^{(k)}$, which proves the existence of the $k$-point correlation function for $\Sineb$.
\end{proof}

\def\cprime{$'$}

{{  
\bigskip
  \footnotesize

\noindent Yahui Qu, \textsc{Department of Mathematics, University of Wisconsin -- Madison, 480 Lincoln Dr,
Madison, WI 53706},  \texttt{yqu45@wisc.edu}

  \medskip

\noindent Benedek Valk\'o, \textsc{Department of Mathematics, University of Wisconsin -- Madison, 480 Lincoln Dr,
Madison, WI 53706}, \texttt{valko@math.wisc.edu}

}}

\end{document}